\documentclass[a4paper, 10pt]{article}
\usepackage{tocloft}
\usepackage{comment}
\usepackage[margin = 1in]{geometry}
\usepackage{aliascnt}

\usepackage{parskip}
\usepackage{nomencl}

\usepackage{fancyhdr} 

\tocloftpagestyle{fancy}
\usepackage{ucs}
\usepackage[T1]{fontenc}
\usepackage{amsmath}
\usepackage{listings}
\usepackage{amssymb} 
\usepackage{amsthm}
\usepackage{amsfonts}
\usepackage{graphicx}
\usepackage{microtype}
\usepackage{mathdots}
\usepackage{mathtools}
\usepackage[shortlabels]{enumitem}
\usepackage{yhmath}
\usepackage{marginnote}
\usepackage{mathrsfs}
\usepackage{bbm}

\usepackage{tikz}
\usepackage{tikz-cd}

\definecolor{linkgreen}{RGB}{33,100,60}
\definecolor{chilligreen}{RGB}{33,100,60}
\definecolor{chillired}{RGB}{140,00,10} 
\definecolor{chillidarkred}{RGB}{120,5,15}

\usepackage[colorlinks=true, citecolor = chilligreen, urlcolor = chillired, backref=page]{hyperref}

\usepackage[nameinlink,capitalise]{cleveref}
\renewcommand*{\backref}[1]{}
\renewcommand*{\backrefalt}[4]{%
	\ifcase #1%
	\or (Cited on page~#2.)%
	\else (Cited on pages~#2.)%
	\fi%
}

\theoremstyle{definition}

\newtheorem{definition}{Definition}[section]

\newaliascnt{question}{definition}
\newtheorem{question}[question]{Question}
\aliascntresetthe{question}

\crefname{question}{Question}{Questions}

\newaliascnt{construction}{definition}
\newtheorem{construction}[construction]{Construction}
\aliascntresetthe{construction}
\crefname{construction}{Construction}{Constructions}

\newaliascnt{observation}{definition}
\newtheorem{observation}[observation]{Observation}
\aliascntresetthe{observation}
\crefname{observation}{Observation}{Observations}

\newaliascnt{conjecture}{definition}

\aliascntresetthe{conjecture}
\crefname{conjecture}{Conjecture}{Conjectures}

\newaliascnt{lemma}{definition}
\newtheorem{lemma}[lemma]{Lemma}
\aliascntresetthe{lemma}
\crefname{lemma}{Lemma}{Lemmas}

\newaliascnt{fact}{definition}

\aliascntresetthe{fact}
\crefname{fact}{Fact}{Facts}

\newaliascnt{setting}{definition}

\aliascntresetthe{setting}
\crefname{setting}{Setting}{Settings}

\newaliascnt{notation}{definition}
\newtheorem{notation}[notation]{Notation}
\aliascntresetthe{notation}
\crefname{notation}{Notation}{Notations}

\newaliascnt{remark}{definition}
\newtheorem{remark}[remark]{Remark}
\aliascntresetthe{remark}
\crefname{remark}{Remark}{Remarks}

\newaliascnt{corollary}{definition}
\newtheorem{corollary}[corollary]{Corollary}
\aliascntresetthe{corollary}
\crefname{corollary}{Corollary}{Corollaries}

\newaliascnt{theorem}{definition}
\newtheorem{theorem}[theorem]{Theorem}
\aliascntresetthe{theorem}
\crefname{theorem}{Theorem}{Theorems}

\newaliascnt{proposition}{definition}
\newtheorem{proposition}[proposition]{Proposition}
\aliascntresetthe{proposition}
\crefname{proposition}{Proposition}{Propositions}

\newaliascnt{example}{definition}
\newtheorem{example}[example]{Example}
\aliascntresetthe{example}
\crefname{example}{Example}{Examples}

\newaliascnt{recollection}{definition}

\aliascntresetthe{recollection}
\crefname{recollection}{Recollection}{Recollections}

\newaliascnt{folklore}{definition}
 
\aliascntresetthe{folklore}
\crefname{folklore}{assumption}{assumptions} 
\Crefname{folklore}{Assumption}{Assumptions}

\newaliascnt{assumption}{definition}

\aliascntresetthe{assumption}
\crefname{assumption}{Assumption}{Assumptions}

\newcommand{\Addresses}{{
  \bigskip
  \footnotesize

  R.~Quinn, \textsc{Utrecht Geometry Center, Universiteit Utrecht, The Netherlands}\\ \nopagebreak
  \textit{E-mail address}: \texttt{r.quinn@uu.nl}

  \medskip

  Q.~Zhu, \textsc{Max Planck Institute for Mathematics, Bonn, Germany}\\ \nopagebreak
  \textit{E-mail address}: \texttt{qzhu@mpim-bonn.mpg.de}

}}

\newtheorem{mainthm}{Theorem}

\newtheorem*{theorem*}{Theorem}
\newtheorem*{example*}{Example}
\newtheorem*{question*}{Question}
\newtheorem*{problem*}{Problem}
\newtheorem*{thmG*}{Theorem G}

\newcommand{\tb}{\textcolor{chillired}}

\newcommand{\B}{\mathrm B}
\newcommand{\C}{\mathscr C}

\newcommand{\Sc}{\mathcal S}

\let\OriginalO\O
\DeclareRobustCommand{\Oslash}{{\OriginalO}}
\renewcommand{\O}{\mathscr{O}}

\DeclareMathAlphabet{\mathbbold}{U}{bbold}{m}{n}

\newcommand{\E}{\mathbb E}  
\newcommand{\F}{\mathbb F}
\newcommand{\R}{\mathbb R}
\renewcommand{\S}{\mathbb S}
\newcommand{\N}{\mathbb N}
\newcommand{\Z}{\mathbb{Z}}

\usepackage{bbm}
\newcommand{\one}{\mathbbm 1}

\newcommand{\tmf}{\operatorname{tmf}}
\newcommand{\ku}{\operatorname{ku}}
\newcommand{\ko}{\operatorname{ko}}

\newcommand{\KU}{\operatorname{KU}}

\newcommand{\MU}{\operatorname{MU}}
\newcommand{\MUP}{\operatorname{MUP}}
\newcommand{\MW}{\operatorname{MW}}
\newcommand{\MUR}{{\MU_{\mathbb{R}}}}
\newcommand{\MUPR}{{\MUP_{\mathbb{R}}}}
\newcommand{\MUG}{{\MU^{(\!(G)\!)}}}

\newcommand{\BPG}{{\BP^{(\!(G)\!)}}}
\newcommand{\BPGm}{{\BP^{(\!(G)\!)}}\langle m\rangle}

\newcommand{\MWR}{{\MW_{\mathbb{R}}}}

\newcommand{\CP}{\mathbb{CP}}
\newcommand{\HP}{\mathbb{HP}}
\newcommand{\CPR}{\mathbb{CP}_{\mathbb{R}}}
\newcommand{\RP}{\mathbb{RP}}
\newcommand{\BO}{\mathrm{BO}}
\newcommand{\THH}{\operatorname{THH}}

\newcommand{\Sp}{\mathrm{Sp}}
\newcommand{\Alg}{\mathrm{Alg}}

\newcommand{\LMod}{\mathrm{LMod}}

\newcommand{\Op}{\mathrm{Op}}

\newcommand{\BP}{\operatorname{BP}}

\newcommand{\BPR}{{\BP_{\mathbb{R}}}}

\newcommand{\BU}{\operatorname{BU}}
\newcommand{\BSU}{\operatorname{BSU}}
\newcommand{\BUR}{{\BU_{\mathbb{R}}}}

\newcommand{\KO}{\operatorname{KO}}
\newcommand{\RO}{\operatorname{RO}}

\newcommand{\uHZ}{\mathrm{H}\underline{\mathbb{Z}}}

\newcommand{\Map}{\operatorname{Map}}

\newcommand{\Pic}{\operatorname{Pic}}
\newcommand{\Th}{\operatorname{Th}}

\newcommand{\Res}{\operatorname{Res}}

\renewcommand{\phi}{\varphi}

\newcommand{\map}{\operatorname{map}}

\newcommand{\gl}{\mathrm{gl}}

\newcommand{\id}{\mathrm{id}}

\newcommand{\Coind}{\operatorname{Coind}}

\newcommand{\res}{\operatorname{res}}
\newcommand{\Nm}{\operatorname{Nm}}

\newcommand{\BUP}{\operatorname{BUP}}

\newcommand{\MOP}{\operatorname{MOP}}
\newcommand{\cofib}{\operatorname{cofib}}
\newcommand{\fib}{\operatorname{fib}}
\newcommand{\MWP}{\mathrm{MWP}}
\newcommand{\MWPR}{\mathrm{MWP}_{\R}}

\newcommand{\kuR}{\ku_{\R}}

\newcommand{\SU}{\operatorname{SU}}

\newcommand{\aut}{\operatorname{aut}}
\newcommand{\BS}{\mathrm{BM}}

\DeclareMathOperator*{\colim}{colim}

\makeatletter
\DeclareRobustCommand{\myuline}[2][0pt]{%
  \ifmmode
    \uline{\hphantom{#2}\kern-#1}%
    \kern#1%
    \mathllap{\mathpalette\my@cont@{#2}}%
  \else
    \uline{\phantom{#2}\kern-#1}%
    \kern#1%
    \llap{\contour{white}{#2}}%
  \fi
}

\newcommand{\ul}{\myuline}
\newcommand{\ol}{\overline}

\usepackage{stmaryrd}

\newcommand{\piccn}{\operatorname{pic}}

\DeclareFontFamily{U}{min}{}
\DeclareFontShape{U}{min}{m}{n}{<-> udmj30}{}

\usepackage[normalem]{ulem}
\usepackage{contour}
\usepackage{mathtools} 

\contourlength{0.8pt}

\newcommand{\my@cont@}[2]{\contour{white}{\mbox{$\m@th#1#2$}}}
\makeatother
\newcommand{\THR}{\mathrm{THR}}

\newcommand{\Snaith}{\mathrm{Snaith}}

\usepackage{fontawesome5}

\newcommand{\notehelper}[3]{\textcolor{#3}{$\blacksquare$}\marginpar{\ifodd\thepage\raggedright\else\raggedleft\fi\color{#3}\tiny \textbf{#2:} #1}}

\begin{document}

	\title{\vspace{-1cm}\textbf{Structured Real Snaith Equivalences}}
    \author{\textsc{Ryan Quinn \& Qi Zhu}}
    \date{} 
	\maketitle

            \hypersetup{linkcolor=chilligreen}

		\begin{abstract}
			 \noindent We give a short proof of the Real Snaith equivalences and multiplicative refinements thereof. The key ingredient is control over structured Real orientations, which we manage through Wilson space theory. In particular, this machinery can be applied to recover an $\E_{2\rho}$-algebra structure on Real Brown--Peterson theory. We apply the Real Snaith theorems to compute $\THR(\KU_{\R})$ and $\THR(\MUPR)$. This requires a norm inverted variant of the Real Snaith theorems, which we prove via the nilpotence theorem.
		\end{abstract}
\hypersetup{linkcolor=chilligreen}

\begin{figure}[ht!]
\centering
\includegraphics[width=95mm]{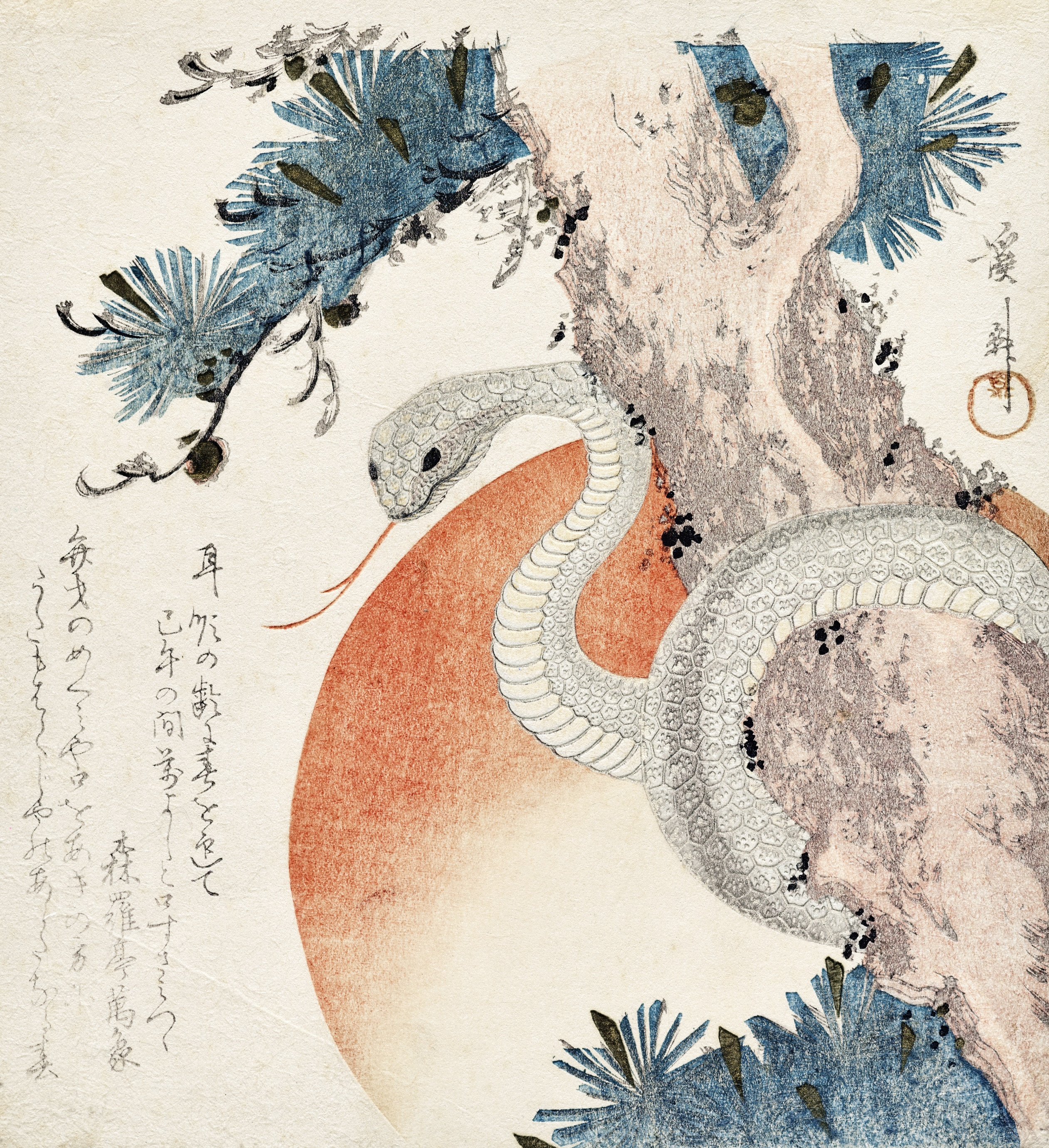}
\end{figure}

\vspace{-0.5cm}


\begingroup
\renewcommand\thefootnote{}
\footnote{\emph{Snake Coiled around a Pine Tree before the Rising Sun} -- Keisai Eisen, 1821.}%
\footnote{\emph{Date}: \today}
\addtocounter{footnote}{-2}
\endgroup
\thispagestyle{empty} 
\setcounter{tocdepth}{2}    
\setcounter{secnumdepth}{4} 
\tableofcontents

\thispagestyle{empty}
\newpage

\section{Introduction}
\label{section: intro}
The periodic complex $K$-theory spectrum $\KU$ and the periodic complex bordism spectrum $\MUP$ play central roles in stable homotopy theory due to their intimate connection to geometry. Indeed, Atiyah and Hirzebruch introduced $\KU$ through complex vector bundles with their direct sums \cite{atiyahHirzebruch1959riemannroch, atiyahHirzebruch1961vectorbundles}, while the homotopy groups of $\MUP$ are essentially represented by stably complex manifolds up to complex bordism due to Thom \cite{thom1954thomspace}. There is another link to geometry provided by theorems of Snaith \cite{snaith1979algebraiccobordism}:
\[ \KU \simeq \Sigma_+^{\infty} \CP^{\infty}[\beta^{-1}] \quad \text{and} \quad \MUP \simeq \Sigma_+^{\infty} \BU[\beta^{-1}], \]
where $\beta$ is the Bott class, a generator in $\pi_2$. So $\KU$ and $\MUP$ are connected to the geometry of line bundles and vector bundles.

Structured refinements of these equivalences suggest a study of the multiplicative structures of $\KU$ and $\MUP$ through the tensor product of vector bundles. However, while $\KU \simeq \Sigma_+^{\infty} \CP^{\infty}[\beta^{-1}]$ is an equivalence of $\E_{\infty}$-algebras, Hahn and Yuan surprisingly proved that $\MUP \simeq \Sigma_+^{\infty} \BU[\beta^{-1}]$ is not an equivalence of $\E_5$-algebras \cite{hahnYuan2020exotic}, so that Snaith's equivalence is not feasible to studying the $\E_{\infty}$-algebra structure of $\MUP$. Nonetheless, Hahn--Yuan exhibit an equivalence of $\E_2$-algebras, allowing us to study the $\E_2$-algebra structure through this viewpoint.

The complex numbers take a distinguished role in this story and it is possible to incorporate the geometry of complex conjugation. By framing the equivalences in terms of $C_2$-equivariant homotopy theory, we arrive at Real candidates of the Snaith theorems studying Atiyah's $K$-theory with Reality $\KU_{\R}$ and the Araki--Landweber periodic Real bordism spectrum $\MUPR$. By $\CP^{\infty}_{\R}$ and $\BU_{\R}$ we denote $C_2$-spaces refining $\CP^{\infty}$ and $\BU$ via the complex conjugation action. The main theorem of this article consists of the Real Snaith equivalences.

Let $\rho$ be the regular representation for the group $C_2$. We denote by $\E_{\rho}$ the corresponding little disk $C_2$-operad and by $\E_{\infty}^{C_2}$ the terminal $C_2$-operad. An $\E_{\infty}^{C_2}$-ring spectrum is also called $C_2$-$\E_{\infty}$-ring spectrum and is informally an $\E_{\infty}$-algebra together with all norm multiplications and suitable coherences.

\begin{mainthm}[\cref{theorem: Erho Real Snaith}] \label{mainthm: Real Snaith}
    There are equivalences
    \[ \KU_{\R} \simeq \Sigma_+^{\infty} \CPR^{\infty}[\beta_{\R}^{-1}] \quad \text{and} \quad \MUP_{\R} \simeq \Sigma_+^{\infty} \BU_{\R}[\beta_{\R}^{-1}] \]
	of $\E_{\infty}^{C_2}$-algebras resp.~of $\E_{\rho}$-algebras, where $\beta_{\R}$ is a generator in $\pi_{\rho}$.
\end{mainthm}

Disregarding the multiplicative structures, this theorem was already established in the motivic literature by Gepner--Snaith \cite{gepnerSnaith2009motivicspectrarepresentingalgebraic} and Annala--Hoyois--Iwasa \cite{annalaIwasa2025motivicspectrauniversalityktheory, annalaHoyoisIwasa2025algebraiccobordismconnerfloydisomorphism}. Certain multiplicative matters of the motivic Snaith theorem have been studied by Röndigs--Spitzweck--\Oslash stvær \cite{roendigsOliverSpitzweck2010motivicstrict}. Furthermore, Schwede has a Real-global refinement of the $\KU$-Snaith in forthcoming work. Nonetheless, the $\E_{\rho}$-structured equivalence is novel and we demonstrate a new and versatile proof strategy. We will apply the multiplicativity in the computation of $\THR(\KU_{\R})$ and $\THR(\MUPR)$.

Our new short proof is inspired by technology that we developed in \cite{quinnZhu2026multiplicativeequivariantthomspectra}. We observe that $\KU_{\R}$ and $\MUPR$ are examples of strongly even $C_2$-spectra in the sense of Hill--Meier \cite{hillmeier2017}. These enjoy the property that equivalences are detected on underlying non-equivariant spectra. So after verifying that all of the involved $C_2$-spectra are strongly even (\cref{prop: strongly even}), it suffices to lift the non-equivariant Snaith equivalences to $C_2$-equivariant maps.

Such a program was precisely carried out in our previous work \cite{quinnZhu2026multiplicativeequivariantthomspectra}, but we must expand on it to be able to control the periodic object $\MUPR$. We realize this through a periodic version of Wilson space theory and simultaneously use this as an opportunity to improve the structured orientations that we produced in \cite{quinnZhu2026multiplicativeequivariantthomspectra}.

\subsubsection*{Structured orientations through Wilson spaces}
The key designer spectra in this story are Thom spectra on Wilson spaces. Consider the Conner--Floyd orientation $\gamma \colon \MU \to \ku$. Then, we may construct the Thom spectra (\cref{construction: MW MWP})
\begin{align*}
        \tb{\MW} &\coloneqq \Th\left(\Omega^\infty \Sigma^2\MU \xrightarrow{\Omega^\infty \Sigma^2 \gamma} 
    \Omega^\infty \Sigma^2\ku \simeq 
    \BU \xrightarrow{\ J \ } \Pic(\Sp)\right)
    \\ \tb{\MWP} &\coloneqq \Th\left(\Omega^\infty \MU \xrightarrow{\Omega^\infty \gamma} 
    \Omega^\infty \ku \simeq 
     \BUP \xrightarrow{\ J \ } \Pic(\Sp)\right),
\end{align*}
where the non-periodic version $\MW$ was first studied by Hahn--Raksit--Wilson in their work on the even filtration \cite{hahn2024motivicfiltrationtopologicalcyclic}. By showing that every even resp.~even periodic $\E_{\infty}$-ring spectrum admits an $\E_{\infty}$-algebra map from $\MW$ resp.~$\MWP$ via Thom spectrum theory (\cref{theorem: commutative maps from MW and MWP}) and splitting off $\MU$ resp.~$\MUP$ in a structured way (\cref{corollary: idempotents MW and MWP}), we are able to produce structured orientations:

\begin{mainthm}[\cref{theorem: E4 or E6 orientations}]
    Let $E$ be an $\E_{\infty}$-ring spectrum.
    \begin{enumerate}[(i)]
        \item If $E$ is even, then there exists an $\E_4$-map $\MU \to E$.
        \item If $E$ is even periodic, then there exists an $\E_6$-map $\MUP \to E$.
    \end{enumerate}
    In particular, any even periodic $\E_{\infty}$-ring admits an $\E_6$-$\MU$-orientation.
\end{mainthm}

It is possible to provide such orientations directly with a similar argument without adhering to Wilson space theory. We are going this route since the proof ingredients are necessary to study a $C_2$-equivariant variant. 

Indeed, after defining Real variants $\MWR$ and $\MWPR$ we are able to use lifting results from  \cite{quinnZhu2026multiplicativeequivariantthomspectra}, to obtain an analogous result in the $C_2$-equivariant setting. This allows us to apply structured Thom isomorphisms and is the key ingredient to obtain full control on structured Real orientations \cite[Theorem F]{quinnZhu2026multiplicativeequivariantthomspectra}. 

The suitable Real generalization of evenness is Hill--Meier's notion of strongly even $C_2$-spectra \cite[Definition 3.1]{hillmeier2017}. If a strongly even $E \in \Alg_{\E_{\infty}}(\Sp^{C_2})$ moreover satisfies $\Sigma^{\rho} E \simeq E$ as $E$-modules, then we call it strongly even periodic (\cref{definition: strongly even periodic}). 

We are able to improve on results in \cite{quinnZhu2026multiplicativeequivariantthomspectra} and obtain:

\begin{mainthm}[\cref{theorem: MUR MUPR lifting result}] \label{mainthm: lifting}
    Let $E$ be an $\E_{\infty}^{C_2}$-ring spectrum.
    \begin{enumerate}[(i)]
        \item If $E$ is strongly even, then any $\E_{2n}$-map $\MU \to \Res_e^{C_2} E$ lifts uniquely to an $\E_{n\rho}$-map $\MUR \to E$ for $n \leq 2$.
        \item If $E$ is strongly even periodic, then any $\E_{2n}$-map $\MUP \to \Res_e^{C_2} E$ lifts uniquely to an $\E_{n\rho}$-map $\MUPR \to E$ for $n \leq 3$.
    \end{enumerate}
\end{mainthm}

The periodic version in (ii) is new and previously we had only established (i) for $n \leq 1$ in \cite{quinnZhu2026multiplicativeequivariantthomspectra}. This result produces structured lifts of many orientations of interest. By formal means we can furthermore produce versions for larger groups $G \geq C_2$. Necessary for such a procedure is the functor $\Coind_{C_2}^G \colon \Op_{C_2, \infty} \to \Op_{G, \infty}$ from the $\infty$-category of $C_2$-$\infty$-operads to the $\infty$-category of $G$-$\infty$-operads, which is the right adjoint of $\Res_{C_2}^G \colon \Op_{G, \infty} \to \Op_{C_2, \infty}$.

We write $\MUG \coloneqq N_{C_2}^G \MUR,  \MUP^{( \! ( G ) \! )} \coloneqq N_{C_2}^G \MUPR$ and $\BPG \coloneqq N_{C_2}^G \BPR$. Let us state a number of examples in the general version for larger groups. 

\begin{mainthm}[\cref{theorem: BPR E2rho}, \cref{corollary: higher group structure examples}]
    Let $G \geq C_2$ be a finite group. Let $n \geq 1$.
    \begin{enumerate}[(i)]
        \item Suppose that $G \leq \mathbb{G}_n$ is a finite subgroup of the extended Morava stabilizer group. There exist $\Coind_{C_2}^G \E_{3\rho}$-algebra maps $\MUP^{( \! ( G ) \! )} \to E_n$ to Lubin--Tate theory.
        \item Let $G = (\Z/n)^{\times}$. Then, the Hirzebruch level-$n$-genera $\MUR \to \tmf_1(n)$ induce $\Coind_{C_2}^G \E_{2\rho}$-algebra maps $\MU^{( \! ( G ) \! )} \to \tmf_1(n)$.
        \item The $G$-spectrum $\BP^{( \! ( G ) \! )}$ admits a $\Coind_{C_2}^G \E_{2\rho}$-algebra structure.
    \end{enumerate}
    In particular, the Real Brown--Peterson spectrum $\BPR$ admits an $\E_{2\rho}$-algebra structure.
\end{mainthm}

In \cite{quinnZhu2026multiplicativeequivariantthomspectra} we already obtained versions of these orientations with a $\Coind_{C_2}^G \E_{\rho}$-structure. There, we relied on Chadwick--Mandell's result saying that every complex orientation lifts uniquely to an $\E_2$-complex orientation \cite{chadwickmandell}. This allowed us to immediately apply the lifting result. On the other hand, not every complex orientation lifts to an $\E_4$-complex orientation, e.g.~it is still not known whether the Quillen idempotent $\MU \to \MU$ lifts to an $\E_4$-algebra map. As such, we need to input strong non-equivariant theorems first to obtain specific $\E_4$-lifts before lifting them to the equivariant setting via \cref{mainthm: lifting}.

For example, the $\E_{2\rho}$-algebra structure on $\BPR$ is obtained by lifting the $\E_4$-idempotent on $\MU_{(2)}$ provided by Basterra--Mandell \cite{basterraMandell2013BP}. Here, we want to clarify that the $\E_{2\rho}$-algebra structure on $\BPR$ is also established in concurrent work by Carrick--Hill--Stewart and the authors, which will appear in forthcoming work (\cref{remark: chqs}). Therefore, we provide another proof of this result. 

\subsubsection*{Real Snaith equivalences \& Real topological Hochschild homology}

By writing down a lift of the $\E_{\infty}$-Snaith equivalence $\KU \simeq \Sigma_+^{\infty} \CP^{\infty}[\beta^{-1}]$ and by lifting the $\E_2$-Snaith equivalence $\MUP \to \Sigma_+^{\infty}\BU[\beta^{-1}]$ from Hahn--Yuan, we thus prove the structured Real Snaith equivalences with the outline suggested after \cref{mainthm: Real Snaith}. So we prove
\[ \KU_{\R} \simeq \Sigma_+^{\infty} \CPR^{\infty}[\beta_{\R}^{-1}] \quad \text{and} \quad \MUP_{\R} \simeq \Sigma_+^{\infty} \BU_{\R}[\beta_{\R}^{-1}] \]
of $\E_{\infty}^{C_2}$-algebras resp.~of $\E_{\rho}$-algebras.

Our main application of $C_2$-equivariant multiplicative structures lies in Real trace theory \cite{hesselholtMadsen2015Realalgebra, dottoMoiPatchkoriaReeh2021THR}. A $C_2$-equivariant refinement of topological Hochschild homology is Real topological Hochschild homology defined by 
\[ \THR(R) \coloneqq R \otimes_{N_e^{C_2} \Res_e^{C_2} R} R \] 
for $C_2$-ring spectra $R$ with $N_e^{C_2}\Res_e^{C_2} R$-$N_e^{C_2}\Res_e^{C_2}R$-bimodule structures, which is e.g.~provided by an $\E_{\sigma}$-algebra structure. Here, $\sigma$ is the $C_2$-sign representation. This is the entry point in approaching Real algebraic $K$-theory, a $C_2$-equivariant refinement of algebraic $K$-theory that furthermore encodes the information of Hermitian $K$-theory \cite{hesselholtMadsen2015Realalgebra, 9authors2023hermitian1, 9authors2025hermitian2, gabe2025realsyntomiccohomology, 9authors2026hermitianktheorystableinftycategories}.

A central computation in the theory of $\THH$ is Bökstedt's theorem, i.e.~the computation of $\THH(\F_p)$, which can be stated as $\THH(\F_p) \simeq \F_p \otimes \Sigma_+^{\infty} \Omega S^3$. As such, a similar level of importance should be given to its Real analog $\THR(\ul{\F}_p) \simeq  \ul{\F}_p \otimes \Sigma_+^{\infty} \Omega S^{1+\rho}$, which was proven by Dotto--Moi--Patchkoria--Reeh \cite{dottoMoiPatchkoriaReeh2021THR}. Using the same strategy as for the Real Snaith theorems, we prove a more structured version of the Real Bökstedt theorem by lifting an $\E_2$-equivalence by Bayındır--Moulinos \cite{bayindirMoulinos2022KTHHFp}.

\begin{mainthm}[\cref{theorem: structured Real Boekstedt}]
There is an equivalence $\THR(\ul{\F}_p) \simeq \ul{\F}_p \otimes \Sigma_+^{\infty}\Omega S^{1+\rho}$ of $\E_{\rho}$-algebras.
\end{mainthm}

This demonstrates that our techniques in proving the Real Snaith theorems are versatile enough to be applicable in various frameworks.

Our main application of the Real Snaith equivalences lies in the computation of $\THR(\KU_{\R})$ and $\THR(\MUP_{\R})$.

\begin{mainthm}[\cref{theorem: THR KUR}] \label{mainthm: THR KUR and THR MUPR}
    There are equivalences 
    \[ \THR(\KU_{\R}) \simeq \KU_{\R} \otimes \Sigma_+^{\infty} \mathrm{B}^\rho \mathrm{U}_{\R}(1) \quad \text{resp.} \quad \THR \left(\MUP_{\R} \right) \simeq \MUP_{\R} \otimes \Sigma_+^{\infty} \B^{\rho}\mathrm{U}_{\R} \] of $\E_{\infty}^{C_2}$-algebras resp.~of $\E_1$-algebras.
\end{mainthm}

Non-equivariantly, such a computation using Snaith's theorem was first carried out by Stonek and Rasekh--Stonek--Valenzuela \cite{stonek2020thhku, rasekhStonekValenzuela2022thom}, who used that $\THH$ commutes with inverting elements and that the $\THH$ of spherical group rings was well-known. In fact, their methods combined with the real Snaith theorem (\cref{corollary: real Snaith}) allows us to give a new formula for $\THH(\MOP)$ in \cref{prop: THH MOP}. 

The $\THR$ of spherical group rings (with anti-involution) is also established -- however, more care is needed in commuting the localization past $\THR$. 

\begin{mainthm}[\cref{prop: THR inversion}] \label{mainthm: THR inverted}
    Let $R$ be an $\E_{\infty}^{C_2}$-ring spectrum and $V$ be a $C_2$-representation. Suppose that $\ol{x} \in \pi_V^{C_2}(R)$ and $x = \Res_e^{C_2}\ol{x} \in \pi_{|V|}^e(R)$. Then, there is an equivalence
    \[ \THR \left(R[(\Nm_e^{C_2}x)^{-1}] \right) \simeq \THR(R)[(\Nm_e^{C_2}x)^{-1}] \]
    of $\E_{\infty}^{C_2}$-algebras.
\end{mainthm}

In particular, the Real Snaith equivalences do not naively lend themselves to a $\THR$ computation since the Real Bott element $\beta_{\R}$ is not of the form required in \cref{mainthm: THR inverted}. Our solution is to improve on the Real Snaith theorems by providing the following norm inverted variants.

\begin{mainthm}[\cref{corollary: norm inverted snaith}]
    There are equivalences
    \[ \KU_{\R} \simeq \Sigma_+^{\infty} \CP^{\infty}_{\R}[\Nm_e^{C_2}(\beta)^{-1}]  \quad \text{and} \quad \MUP_{\R} \simeq \Sigma_+^{\infty} \BUR[\Nm_e^{C_2}(\beta)^{-1}] \]
    of $\E_{\infty}^{C_2}$- resp.~$\E_{\rho}$-algebras.
\end{mainthm}

In proving this theorem, we do not explicitly compute $\Nm_e^{C_2}(\beta)$, but rather proceed indirectly via an application of the nilpotence theorem on $C_2$-geometric fixed points. Another proof ingredient is the Real Snaith theorem itself.

Combining these equivalences with the $\THR$ computation of spherical group rings allows us to carry out the computation in \cref{mainthm: THR KUR and THR MUPR}.

\subsection*{Outline}
The remainder of the article is organized as follows:

In \cref{section: structure real orientations} we use Wilson space theory to produce $\E_4$- resp.~$\E_6$-orientations of even resp.~even periodic $\E_{\infty}$-ring spectra. We lift these to a $C_2$-equivariant setting and prove a lifting theorem for structured orientations. The rest of the article can be read independently of this section by accepting \cref{theorem: MUR MUPR lifting result}.

In \cref{section: Real Snaith Equivalences} we prove the structured Real Snaith theorems and their norm-inverted variants. We spell out the real Snaith theorems as consequences.

In \cref{section: computations in THR} we analyze when one can commute $\THR$ with inverting an element. We then use this to compute $\THR(\KU_{\R})$ and $\THR(\MUP_{\R})$. We also prove an $\E_{\rho}$-version of the Real Bökstedt theorem.

\subsection*{Notations \& Conventions}
We collect some notational conventions here that we will use throughout the whole article.
\begin{enumerate}[(1)]
    \item Throughout, $G$ will always be a finite group.
    \item We denote by $\rho$ and \(\sigma\) the regular and sign representations of $C_2$ respectively.
    \item Let $X,Y \in \Sp_G$. We  denote by $\ul{\map}_G(X, Y)$ the mapping $G$-spectrum refining the mapping $G$-space $\ul{\Map}_G(X, Y)$. We also write $\ul{\map}_{\ul{\Sp}_G}(X, Y)$ and $\ul{\Map}_{\ul{\Sp}_G}(X, Y)$.
    \item The $G$-$\infty$-operad $\E_{\infty}^G$ is the terminal $G$-$\infty$-operad. So $\E_{\infty}^G$-algebras admit all norms -- they are also known as normed algebras, ultracommutative algebras, or \(G\)-\(\E_\infty\)-algebras.
    \item There is a restriction functor of equivariant $\infty$-operads $\Res_H^G \colon \Op_{G, \infty} \to \Op_{H, \infty}$ which admits a right adjoint $\Coind_H^G \colon \Op_{H, \infty} \to \Op_{G, \infty}$, the coinduction functor. See \cite[Section 1.3.1]{stewart2025tensorproductsequivariantcommutative} and \cite[Remark 2.2.5]{quinnZhu2026multiplicativeequivariantthomspectra} for a discussion in the language of parametrized higher categories.
\end{enumerate}

\subsection*{Acknowledgements}
We thank Christian Carrick, Markus Hausmann,  Kaif Hilman, Ryomei Iwasa, Ishan Levy, Lennart Meier, Lucas Piessevaux and Stefan Schwede for helpful and encouraging discussions. 
RQ is funded by the NUI Travelling Studentship. QZ is supported by the Max Planck Institute for Mathematics (MPIM) in Bonn and is thankful
for its financial support and for providing  conducive working environments. QZ furthermore thanks Utrecht University for its hospitality, where a portion of this article was written.

\section{Structured Real orientations through Wilson spaces}
\label{section: structure real orientations}

\subsection{Recollection on lifting equivariant maps}

We recall our lifting theorems from \cite{quinnZhu2026multiplicativeequivariantthomspectra}, since they will be applied in the text and motivate our discussion of Wilson space theory. Let us denote by $\rho$ the regular representation of the group $C_2$. Our first general lifting result is the following.

\begin{theorem}[{\cite[Theorem D]{quinnZhu2026multiplicativeequivariantthomspectra}}]
    Let \(E \in \Sp^{C_2}\) be strongly even and $X \in \Sp^{C_2}$. Suppose the following:
    \begin{enumerate}[(i)]
        \item The $C_2$-spectrum \(X\) is slice bounded below.
        \item There is an indexing set $I$ such that $\uHZ\otimes X\simeq \uHZ\otimes \bigoplus_{i\in I}S^{n_i\rho}$ for certain $n_i \in \N$ where every degree appears finitely often.
    \end{enumerate}
    Then, every map $\Res_e^{C_2} X \to \Res_e^{C_2}E$ lifts uniquely to a $C_2$-equivariant map $X \to E$.
\end{theorem}

Using this general lifting result, we can deduce a structured lifting result. Concerning equivariant multiplications, let us recall that every $C_2$-representation $V$ gives rise to an equivariant little disk operad $\E_V$. We denote by $\E_{\infty}^{C_2}$ the terminal $C_2$-operad, its algebras are also called $C_2$-$\E_{\infty}$-algebras in the literature. Informally, those are $\E_{\infty}$-algebras together with all norm multiplications and possible coherences. Furthermore, we will write $\Th_{C_2}$ for the $C_2$-equivariant Thom spectrum functor, for example constructed through parametrized colimits \cite{hahn2024equivariantnonabelianpoincareduality, quinnZhu2026multiplicativeequivariantthomspectra}.

\begin{theorem}[{\cite[Theorem E]{quinnZhu2026multiplicativeequivariantthomspectra}}] \label{theorem: chili thom lifting}
    Let $R$ be an $\E_{\infty}^{C_2}$-ring spectrum and $E$ be a strongly even $\E_{\infty}^{C_2}$-algebra in $\ul{\LMod}_R$. Suppose that $X$ is an $n\rho$-loop space with a map $f \colon X \to \ul{\Pic}_{C_2}(R)$ of $n\rho$-loop spaces. Suppose the following:
    \begin{enumerate}[(i)]
        \item There exists an \(\E_{n\rho}\)-\(R\)-algebra map \(\Th_{C_2}(f)\to E\).
        \item There is an indexing set $I$ such that $\uHZ\otimes \mathrm{B}^{n\rho}X\simeq \uHZ\otimes \bigoplus_{i\in I}S^{n_i\rho}$ for certain $n_i \in \N$ where every degree appears finitely often.
    \end{enumerate} 
    Then, every $\E_{2n}$-map $\Res_e^{C_2} \Th_{C_2}(f) \to \Res_e^{C_2}E$ lifts uniquely to an $\E_{n\rho}$-map $\Th_{C_2}f \to E$.
\end{theorem}

The main crux to lifting from Thom spectra consists of the existence of a structured orientation (\cref{theorem: chili thom lifting}\textcolor{chilligreen}{(i)}), since it allows us to perform structured Thom isomorphisms. So structured Real orientations $\MU_{\R} \to E$ depend crucially on the existence of a structured map $\MU_{\R} \to E$. The theory of Wilson spaces is what allows us to access such orientations. 

In \cite{quinnZhu2026multiplicativeequivariantthomspectra} we laid out such a theory for $\E_{\rho}$-orientations. The goal of the current section is to expand on this theory to $\E_{2\rho}$-orientations and to give a periodic variant of it. This will be a theory of $\E_{3\rho}$-$\MUPR$-orientations and it will be relevant to access the Real Snaith theorem.

\subsection{Structured orientations through Wilson spaces}
Inspired by \cite[Theorem 3.2.10]{hahn2024motivicfiltrationtopologicalcyclic} we will provide new structured complex orientations of even (periodic) $\E_{\infty}$-ring spectra. The key will be to manipulate the remarkable obstruction-theoretic features of the Wilson spaces $W_{2i} \coloneqq \Omega^{\infty}\Sigma^{2i}\MU$ that were first studied by Steve Wilson in his PhD thesis \cite{Wilson1973}.

\begin{construction} \label{construction: MW MWP}
    Let \(\gamma\colon \MU\to \ku\) be a map that is surjective on homotopy.\footnote{Such a map exists, and is e.g.~realized by the Conner--Floyd orientation.}
    We consider the Thom spectra
    \begin{align*}
        \tb{\MW} &\coloneqq \Th\left(\Omega^\infty \Sigma^2\MU \xrightarrow{\Omega^\infty \Sigma^2 \gamma} 
    \Omega^\infty \Sigma^2\ku \simeq 
    \BU \xrightarrow{ \ J \ } \Pic(\Sp)\right),
    \\ \tb{\MWP} &\coloneqq \Th\left(\Omega^\infty \MU \xrightarrow{\Omega^\infty \gamma} 
    \Omega^\infty \ku \simeq 
    \BUP \xrightarrow{\ J \ } \Pic(\Sp)\right).
    \end{align*}
    The spectrum $\MW$ was introduced by Hahn--Raksit--Wilson in \cite[Definition 3.2.9]{hahn2024motivicfiltrationtopologicalcyclic}, while $\MWP$ is a natural periodic version of it.
\end{construction}

Recall that an even $\E_{\infty}$-ring spectrum $E$ is called \emph{even periodic} if $\Sigma^2 E \simeq E$ of $E$-modules.

\begin{theorem} \label{theorem: commutative maps from MW and MWP}
    Let $E$ be an $\E_{\infty}$-ring spectrum.
    \begin{enumerate}[(i)]
        \item If $E$ is even, then there exists an $\E_{\infty}$-map $\MW \to E$.
        \item If $E$ is even periodic, then there exists an $\E_{\infty}$-map $\MWP \to E$.
    \end{enumerate}
\end{theorem}

\begin{proof}
    For expository purposes, we will recite \cite[Proof of Theorem 3.2.10]{hahn2024motivicfiltrationtopologicalcyclic} and \cite[Theorem 5.10]{gabe2025realsyntomiccohomology} in (i), even though the proof technique of (ii) also applies for (i). The same technique will not be feasible to conclude (ii).
    \begin{enumerate}[(i)]
        \item The spectrum $\MU$ is connective and has an even cell decomposition,\footnote{The even cell decomposition comes from the computation $H_*\MU \cong \Z[b_1, b_2, \cdots]$. This also reveals why the same argument does not work for $\ku$ and $\Omega^{\infty} \Sigma^2 \ku \simeq \BU$, as $H_* \ku$ is considerably more complicated.} so Thomifying $\Omega^{\infty} \Sigma^2(-)$ of that cell decomposition yields an even $\E_{\infty}$-cell decomposition of $\MW$ with a single $0$-cell. The obstructions to the extensions
    \begin{center}
        \begin{tikzcd}
            \S \arrow[r] \arrow[d] & E
            \\ \MW \arrow[ur, dashed]
        \end{tikzcd}
    \end{center}
    as $\E_{\infty}$-maps lie in odd homotopy groups of $E$, which vanish.
    \item Since $\MWP$ is unbounded, it seems hard to mimic cell attaching arguments here. We will employ a Thom spectrum argument instead.
    
    An $\E_{\infty}$-ring map $\MWP \to E$ is the same datum as a nullhomotopy of
    \begin{center}
        \begin{tikzcd}
            \MU \arrow[r, "\gamma"] & \ku \arrow[r, "J"] & \operatorname{pic}{\S} \arrow[r] & \operatorname{pic}{E},
        \end{tikzcd}
    \end{center}
    see e.g.~\cite[Lemma 3.15]{antolinbarthel2019thom}. So we need to show that this composite is nullhomotopic. Consider the fiber sequence
    \begin{center}
        \begin{tikzcd}
            \Sigma \gl_1 E \simeq \tau_{\geq 1 }(\operatorname{pic}{E}) \arrow[r] & \operatorname{pic}{E} \arrow[r] & \pi_0(\operatorname{pic}{E})
        \end{tikzcd}
    \end{center}
    coming from the standard $t$-structure of $\Sp$. We wish to show that our aforementioned composite factors through $\Sigma \gl_1 E \to \operatorname{pic}{E}$. The composite 
    \begin{center}
        \begin{tikzcd}
            \ku \arrow[r, "J"] & \operatorname{pic}{\S} \arrow[r] & \operatorname{pic}{E} \arrow[r] &\pi_0(\operatorname{pic}{E}) 
        \end{tikzcd}
    \end{center}
    factors through $\ku \to \tau_{\leq 0} \ku \simeq \pi_0 \ku$ since $\pi_0(\operatorname{pic}{E})$ is $0$-truncated. The map $\pi_0 \ku \to \pi_0(\operatorname{pic}{E})$ is described by sending a complex vector space $V$ to $\Sigma^{|V|}E$. On the other hand, $\Sigma^{|V|}E \simeq E$ by even periodicity of $E$. So the map $\pi_0 \ku \to \pi_0(\operatorname{pic}{E})$ is the null map.
    
    Thus, there exists a factorization
    \begin{center}
        \begin{tikzcd}
            & & & \Sigma \gl_1 E \arrow[d]
            \\\MU \arrow[r, "\gamma", swap] & \ku \arrow[r, "J", swap] \arrow[urr, bend left, dashed, "\exists"] & \operatorname{pic}{\S} \arrow[r] & \operatorname{pic}{E}
        \end{tikzcd}
    \end{center}
    An Atiyah--Hirzebruch spectral sequence computation shows that $\map_{\Sp}(\MU, \Sigma\gl_1 E)$ is odd, see \cite[Proposition B.2.3]{quinnZhu2026multiplicativeequivariantthomspectra}. So every map $\MU \to \Sigma \gl_1 E$ is nullhomotopic. \qedhere
    \end{enumerate}
\end{proof}

\noindent As we have indicated, \cref{theorem: commutative maps from MW and MWP}(i) was known due to Hahn--Raksit--Wilson \cite{hahn2024motivicfiltrationtopologicalcyclic}. 
However, the $\E_{\infty}$-$\MWP$-orientation required a new argument.

Our next goal is to provide structured maps $\MU \to \MW$ and $\MUP \to \MWP$. To get started on this, we will need a quick lemma.

\begin{lemma} \label{lemma: existence of lifts}
    Consider two maps $f \colon x \to z$ and $g \colon y \to z$ of even connective spectra. Suppose that $H^*(x;A)$ is even for every finitely generated abelian group $A$, that $\pi_*y$ and $\pi_* z$ are of finite type and that $\pi_* g$ is surjective. Then, there exists a lift
    \begin{center}
        \begin{tikzcd}
            & y \arrow[d, "g"]
            \\ x \arrow[r, "f", swap] \arrow[ur, dashed] & z
        \end{tikzcd}
    \end{center}
    of connective spectra.
\end{lemma}

\begin{proof}
    The cofiber $\cofib{g}$ in $\Sp_{\geq 0}$ is formed in $\Sp$. There is an exact sequence
    \begin{center}
        \begin{tikzcd}
            \pi_0 \map_{\Sp}(x, y) \arrow[r] & \pi_0 \map_{\Sp}(x, z) \arrow[r] & \pi_0 \map_{\Sp}(x, \cofib g),
        \end{tikzcd}
    \end{center}
    so it suffices to show $\pi_0 \map_{\Sp}(x, \cofib g) = 0$. On the other hand, $\cofib{g}$ is odd by evenness of $y,z$ and surjectivity of $\pi_* g$. An Atiyah--Hirzebruch spectral sequence style argument e.g.~as in \cite[Proposition B.2.3]{quinnZhu2026multiplicativeequivariantthomspectra} allows us to conclude.
\end{proof}

\begin{proposition} \label{proposition: decomposition of MW and MWP}
    There are equivalences
    \[ \MW \simeq \MU \otimes\, {\Sigma_+^{\infty}\fib(\Omega^{\infty}\Sigma^2 \gamma)} \quad \text{resp.} \quad \MWP \simeq \MUP \otimes \,\Sigma_+^{\infty}\fib(\Omega^{\infty}\gamma) \]
    of $\E_4$- resp.~$\E_6$-ring spectra.
\end{proposition}

\begin{proof}
    Let's consider $\MW$. It suffices to construct an $\E_4$-section of $\Omega^{\infty} \Sigma^2 \gamma \colon \Omega^{\infty} \Sigma^2 \MU \to \BU$. To construct this, we will provide a section of the map
    \[ \Omega^{\infty} \Sigma^6 \gamma \colon \Omega^{\infty} \Sigma^6 \MU \longrightarrow \B^4 \BU \simeq \BU \langle 6 \rangle \]
    of pointed spaces and take $\Omega^4$ afterwards. Consider the diagram
    \begin{center}
        \begin{tikzcd}
            & \Omega^{\infty} \Sigma^6 \MU \arrow[d, "\Omega^{\infty}\Sigma^6 \gamma"]
            \\ \BU \langle 6 \rangle \arrow[r, equal] \arrow[ur, dashed] & \BU \langle 6 \rangle
        \end{tikzcd}
    \end{center}
    We obtain such a lift by \cref{lemma: existence of lifts}. Here, the choice of $\gamma$ such that $\pi_* \gamma$ is surjective crucially enters, and moreover, the even cohomology of $\BU\langle 6 \rangle$ was computed by Stong \cite{stong1963BU} and Singer \cite{singer1968connective}, see also \cite[Section 4.5]{andoHopkinsStrickland2001elliptic}. 

    For $\MWP$ we need to provide an $\E_6$-section of $\Omega^{\infty}\gamma \colon \Omega^{\infty} \MU \to \BUP$. But we have provided a section of
    \[ \Omega^{\infty} \Sigma^6 \gamma \colon \Omega^{\infty} \Sigma^6 \MU \longrightarrow \B^4 \BU \simeq \BU \langle 6 \rangle \]
    above and can take $\Omega^6$ of this.
\end{proof}

\begin{remark} \label{remark: same as in HRW}
    A splitting $\MW \simeq \MU \otimes \Sigma_+^{\infty} \fib(\Omega^{\infty} \Sigma^2 \gamma)$ of $\E_2$-ring spectra was already provided by \cite[Theorem 3.2.10]{hahn2024motivicfiltrationtopologicalcyclic} using the exact same argument. Our only observation is that it can be delooped twice more. Nonetheless, we decided to carefully repeat the proof to make sure that the refinement is valid. In particular, one has to be more careful since $\B^4 \BU \simeq \BU\langle 6 \rangle$ is more delicate than $\B^2 \BU \simeq \BSU$.
\end{remark}

\begin{corollary} \label{corollary: idempotents MW and MWP}
    There is an $\E_4$- resp.~$\E_6$-map 
    \[ \MW \longrightarrow \MW \quad \text{resp.} \quad \MWP \longrightarrow \MWP \] 
    which are idempotents underlying and split off $\MU$ resp.~$\MUP$ as an $\E_4$- resp.~$\E_6$-retract.
\end{corollary}
\begin{proof}
    We apply the previous proposition (\cref{proposition: decomposition of MW and MWP}) and the fact that algebras in $\Sc$ are (multiplicatively) augmented, so we construct the idempotents as
    \begin{align*}
        \MW &\simeq \MU \otimes \Sigma_+^{\infty}\fib(\Omega^{\infty}\Sigma^2 \gamma) \to \MU \otimes \Sigma_+^{\infty}* \to \MU \otimes \Sigma_+^{\infty}\fib(\Omega^{\infty}\Sigma^2 \gamma) \simeq \MW
        \\ \MWP &\simeq \MUP \otimes \Sigma_+^{\infty}\fib(\Omega^{\infty}\gamma) \to \MUP \otimes \Sigma_+^{\infty}* \to \MUP \otimes \Sigma_+^{\infty}\fib(\Omega^{\infty}\gamma) \simeq \MWP
    \end{align*}
    by augmentation followed by the unit.
\end{proof}

Combining the two main observations of this section (\cref{theorem: commutative maps from MW and MWP}, \cref{corollary: idempotents MW and MWP}) we are able to produce structured orientations.

\begin{theorem} \label{theorem: E4 or E6 orientations}
    Let $E$ be an $\E_{\infty}$-ring spectrum.
    \begin{enumerate}[(i)]
        \item If $E$ is even, then there exists an $\E_4$-map $\MU \to E$.
        \item If $E$ is even periodic, then there exists an $\E_6$-map $\MUP \to E$.
    \end{enumerate}
    In particular, any even periodic $\E_{\infty}$-ring admits an $\E_6$-$\MU$-orientation.
\end{theorem}

\begin{proof}
    We construct them as compositions
    \begin{center}
        \begin{tikzcd}
            \MU \arrow[r] & \MW \arrow[r] & E & \text{resp.} & \MUP \arrow[r] & \MWP \arrow[r] & E
        \end{tikzcd}
    \end{center}
    of maps from \cref{theorem: commutative maps from MW and MWP} and \cref{corollary: idempotents MW and MWP}. Since the preferred map $\MU \to \MUP$ is an $\E_{\infty}$-ring map, we also deduce the last sentence of the assertion.
\end{proof}

One could instead also have run an argument as in \cref{theorem: commutative maps from MW and MWP} by considering $6$-fold deloopings to construct these orientations directly. Nonetheless, we will mimic \cref{theorem: commutative maps from MW and MWP} to prove a $C_2$-equivariant analog and \cref{corollary: idempotents MW and MWP} will be necessary in the proof of the $C_2$-equivariant analog (\cref{theorem: MUR MUPR lifting result}). That's why we are taking the route through Wilson spaces.

As an immediate example we obtain a structured orientation on the Lubin--Tate theories constructed by the Goerss--Hopkins--Miller theorem \cite{rezk1998hopkinsmiller, GoerssHopkins2004Moduli, lurie2018sag}.

\begin{corollary}
    Let \(k\) be a perfect field and \(\Gamma_h\) be a height \(h\) formal group law over \(k\). Let \(E(k,\Gamma_h)\) be the associated Lubin--Tate theory. There exists an \(\E_6\)-ring map \(\MUP\to E(k,\Gamma_h)\). 
\end{corollary}

\begin{remark}
    In the case that \(k\) is algebraically closed, Burklund--Schlank--Yuan have already shown that there exists an \(\E_\infty\)-ring map \(\MUP\to E(k,\Gamma_h)\) through the chromatic Nullstellensatz \cite[Corollary 8.13]{burklund2022chromaticnullstellensatz}.
\end{remark}

\noindent 

\begin{remark} \label{remark: orientation of MUP Snaith}
    This article is about Snaith's equivalence $\MUP \simeq \Sigma_+^{\infty}\BU[\beta^{-1}] \eqqcolon \MUP_{\Snaith}$ for a certain Bott class $\beta \in \pi_2(\Sigma_+^{\infty}\BU)$, see \cite{snaith1979algebraiccobordism}. As mentioned in \cref{section: intro}, Hahn--Yuan discovered that this equivalence surprisingly does not refine to an equivalence of $\E_{\infty}$-ring spectra. In fact, they show that there does not exist an $\E_5$-equivalence $\MUP \simeq \MUP_{\Snaith}$, and they provide an $\E_2$-equivalence \cite{hahnYuan2020exotic}.

    Hahn--Yuan then ask whether there exists an $\E_{\infty}$-algebra map $\MU \to \MUP_{\Snaith}$. Our results are thus a first step towards answering this question, namely we provide an $\E_6$-algebra map, in fact there even exists an $\E_6$-algebra map $\MUP \to \MUP_{\Snaith}$ by \cref{theorem: E4 or E6 orientations}.
\end{remark}

Our structured orientations came from structured decompositions of $\MW$ and $\MWP$, so it is natural to ask how much these can be refined.

\begin{question}
    We have constructed $\E_4$- and $\E_6$-decompositions
    \[ \MW \simeq \MU \otimes \Sigma_+^{\infty} \fib(\Omega^{\infty} \Sigma^2 \gamma) \quad \& \quad \MWP \simeq \MUP \otimes \Sigma_+^{\infty} \fib(\Omega^{\infty}\gamma). \]
    Do they admit even more structure?
\end{question}

\noindent Since further deloopings of $\BU\langle 6 \rangle$ do not have cohomology concentrated in even degrees \cite{stong1963BU}, the obstruction-theoretic arguments as written in \cref{proposition: decomposition of MW and MWP} do not work anymore.

\subsection{Structured Real orientations}
In \cite{quinnZhu2026multiplicativeequivariantthomspectra} we developed an obstruction theory for lifting non-equivariant structured orientations to structured equivariant orientations. One of our main ingredients to handle the Thom spectrum $\MU_{\R}$ was a $C_2$-equivariant analog of the Wilson space result \cref{theorem: E4 or E6 orientations}. At the time, we were following Hahn--Raksit--Wilson and were under the impression that $\MW$ only splits off $\MU$ as an $\E_2$-retract (\cref{remark: same as in HRW}). With the realization that an $\E_4$-retract is possible as well as the discussion of a periodic version (\cref{theorem: E4 or E6 orientations}), we are thus able to improve upon our results. As an application we will for example be able to produce an $\E_{2\rho}$-algebra structure on $\BPR$, see \cref{theorem: BPR E2rho}.

The Real analogs of $\MW$ and $\MWP$ from \cref{construction: MW MWP} are the following $C_2$-Thom spectra.

\begin{construction}
    Let \( \gamma_{\R} \colon \MUR\to \kuR\) be a map such that $\pi_{* \rho}\gamma_{\R}$ surjective.\footnote{Such a map exists, e.g.~by lifting the Conner--Floyd orientation, see e.g.~\cite[Corollary 6.2.2]{quinnZhu2026multiplicativeequivariantthomspectra}.}
    We consider the $C_2$-Thom spectra
    \begin{align*}
        \tb{\MWR} &\coloneqq \Th_{C_2}\left(\Omega^\infty \Sigma^{\rho}\MU_{\R} \xrightarrow{\Omega^\infty \Sigma^{\rho} \gamma_{\R}} 
    \Omega^\infty \Sigma^{\rho}\ku_{\R} \simeq 
    \BU_{\R} \xrightarrow{\ J_{\R} \ } \ul{\Pic}(\ul{\Sp}^{C_2})\right),
    \\ \tb{\MWPR} &\coloneqq \Th_{C_2}\left(\Omega^\infty \MU_{\R} \xrightarrow{\Omega^\infty \gamma_{\R}} 
    \Omega^\infty \ku_{\R} \simeq 
    \BUP_{\R} \xrightarrow{\ J_{\R} \ } \ul{\Pic}(\ul{\Sp}^{C_2})\right).
    \end{align*}
    The $C_2$-spectrum $\MWR$ was considered by Angelini-Knoll--Kong--Quigley \cite[Construction 5.8]{gabe2025realsyntomiccohomology} and $\MWPR$ is a natural periodic analog.
\end{construction}

\begin{definition} \label{definition: strongly even periodic}
    Let $R \in \Alg_{\E_{\infty}}(\Sp^{C_2})$. It is said to be \tb{strongly even periodic} if it is strongly even and $\Sigma^{\rho}R \simeq R$ of $R$-modules.
\end{definition}

In particular, the underlying spectrum of a strongly even periodic $C_2$-spectrum is an even periodic spectrum.

\begin{theorem} \label{theorem: commutative maps from MWR and MWPR}
    Let $E$ be an $\E_{\infty}^{C_2}$-ring spectrum.
    \begin{enumerate}[(i)]
        \item If $E$ is strongly even, then there exists an $\E_{\infty}^{C_2}$-map $\MWR \to E$.
        \item If $E$ is strongly even periodic, then there exists an $\E_{\infty}^{C_2}$-map $\MWP_{\R} \to E$.
    \end{enumerate}
\end{theorem}

\begin{proof}
    Part (i) already appeared in \cite[Theorem 5.10]{gabe2025realsyntomiccohomology}. Let us only discuss (ii), whose proof strategy also applies to (i). We essentially mimic our proof in \cref{theorem: commutative maps from MW and MWP}.

    An $\E_{\infty}^{C_2}$-ring map $\MWPR \to E$ is the same datum as a nullhomotopy of
    \begin{center}
        \begin{tikzcd}
            \MUR \arrow[r, "\gamma_{\R}"] & \kuR \arrow[r, "J_{\R}"] & \ul{\operatorname{pic}}(\S) \arrow[r] & \ul{\operatorname{pic}}(E)
        \end{tikzcd}
    \end{center}
    by \cite[Lemma 4.2.7]{quinnZhu2026multiplicativeequivariantthomspectra}, where $\ul{\piccn}$ denotes the corresponding connective $C_2$-spectrum. So we need to show that this composite is nullhomotopic. There is a fiber sequence
    \begin{center}
        \begin{tikzcd}
            \Sigma \gl_1 E \simeq \tau_{\geq 1}(\ul{\operatorname{pic}}(E)) \arrow[r] & \ul{\operatorname{pic}}(E) \arrow[r] & \ul{\pi}_0(\ul{\operatorname{pic}}(E))
        \end{tikzcd}
    \end{center}
    coming from the standard $t$-structure of $\Sp^{C_2}$. The first equivalence is by $\tau_{\geq 1} \ul{\piccn}(E) \simeq \Sigma \gl_1 E$ using the natural equivalence $\Omega \piccn(\C) \simeq \aut_{\C}(\one_{\C})$ for any symmetric monoidal $\infty$-category $\C$, see e.g.~\cite[Section 2.2]{mathewstojanoska2016tmf}. 
    We now wish to show that our aforementioned composite factors through $\Sigma \gl_1 E \to \ul{\operatorname{pic}}(E)$. The composite
    \begin{center}
        \begin{tikzcd}
            \ku_{\R} \arrow[r, "J_{\R}"] & \ul{\piccn}(\S) \arrow[r] & \ul{\piccn}(E) \arrow[r] & \ul{\pi}_0(\ul{\piccn}(E))
        \end{tikzcd}
    \end{center}
    factors through $\kuR \to \tau_{\leq 0} \kuR \simeq \ul{\pi}_0 \kuR$ since $\ul{\pi}_0(\ul{\piccn}(E))$ is $0$-truncated. The map $\ul{\pi}_0 \kuR \to \ul{\pi}_0(\ul{\piccn}(E))$ is now nullhomotopic, as we may check levelwise:
    \begin{itemize}
        \item On $\pi_0^e$ we obtain the $J$-homomorphism $\pi_0 \ku \to \pi_0(\piccn E^e), \ 1 \mapsto \S^2 \otimes E$, which is null by even periodicity of $E^e$.
        \item On $\pi_0^{C_2}$ we obtain the map $\pi_0 \ko \to \pi_0(\piccn E), \ 1 \mapsto \S^{\rho} \otimes E$ which is null by strong even periodicity of $E$. 
        
        Let us justify that this map sends $1$ to $\S^{\rho} \otimes E$. It factors through $\pi_0(\piccn \S_{C_2}) \to \pi_0(\piccn E)$, so it suffices to show that $\pi_0 \ko \to \pi_0(\piccn \S_{C_2})$ sends $1$ to $\S^{\rho}$. Since the Real $J$-homomorphism is an $\E_{\infty}^{C_2}$-algebra map \cite[Appendix A.2]{quinnZhu2026multiplicativeequivariantthomspectra}, the diagram
        \begin{center}
            \begin{tikzcd}
                \pi_0 \ko \arrow[r] & \pi_0(\piccn \S_{C_2}) 
                \\ \pi_0 \ku \arrow[u, "\Nm_{e}^{C_2}"] \arrow[r] & \pi_0(\piccn \S) \arrow[u, "\Nm_{e}^{C_2}", swap] 
            \end{tikzcd}
        \end{center}
        commutes. The norm map $\Nm_{e}^{C_2} \colon \pi_0 \ku \to \pi_0 \ko$ sends $1$ to $2$. Going the other way, the composite is given by $1 \mapsto \S^2 \mapsto \S^{2\rho}$. We deduce $2 \mapsto \S^{2\rho}$. On the other hand, $\pi_0(\piccn \S_{C_2}) \cong \Z \times \Z$ is known to only consist of the $\RO(C_2)$-graded spheres by \cite[Theorem 5]{tomDieck1978homotopy} and \cite[Theorem 0.1]{fauskLewisMay2001picardequivariant}.  Thus, we must have $1 \mapsto \S^{\rho}$.
    \end{itemize}
    Thus, there exists a factorization
    \begin{center}
        \begin{tikzcd}
            & & & \Sigma \gl_1 E \arrow[d]
            \\ \MUR \arrow[r, "\gamma_{\R}", swap] & \kuR \arrow[r, "J_{\R}", swap] \arrow[urr, bend left, "\exists", dashed] & \ul{\piccn}(\S) \arrow[r] & \ul{\piccn}(E)
        \end{tikzcd}
    \end{center}
    By \cite[Theorem 5.3.8]{quinnZhu2026multiplicativeequivariantthomspectra} the shift of the $C_2$-mapping spectrum $\Omega \ul{\map}_{\ul{\Sp}^{C_2}}(\MUR, \Sigma \gl_1 E)$ is strongly even. By Greenlees' gap characterization \cite[Lemma 1.2]{greenleesfour} we conclude
    \[ \pi_0^{C_2} \ul{\map}_{\ul{\Sp}^{C_2}}(\MUR, \Sigma \gl_1 E) \cong \pi_{-1}^{C_2} \Omega \ul{\map}_{\ul{\Sp}^{C_2}}(\MUR, \Sigma \gl_1 E) \cong 0. \]
    Hence, any map $\MUR \to \Sigma \gl_1 E$ is nullhomotopic.
\end{proof}

Similar to the non-equivariant case (\cref{theorem: commutative maps from MW and MWP}), let us stress again that the $\E_{\infty}^{C_2}$-$\MWR$-orientation was already known due to Angelini-Knoll--Kong--Quigley \cite[Theorem 5.10]{gabe2025realsyntomiccohomology}, while the $\E_{\infty}^{C_2}$-$\MWPR$-orientation is really a new result that required a number of equivariant theorems, such as those developed in \cite{quinnZhu2026multiplicativeequivariantthomspectra}.

\begin{proposition} \label{proposition: decomposition of MWR and MWPR}
    There are equivalences
    \[ \MWR \simeq \MUR \otimes \Sigma_+^{\infty}\fib(\Omega^{\infty}\Sigma^{\rho} \gamma_{\R}) \quad \text{resp.} \quad \MWPR \simeq \MUPR \otimes \Sigma_+^{\infty}\fib(\Omega^{\infty}\gamma_{\R}) \]
    of $\E_{2\rho}$- resp.~$\E_{3\rho}$-ring spectra.
\end{proposition}

\begin{proof}
    Provided the $\E_{\infty}^{C_2}$-orientations from \cref{theorem: commutative maps from MWR and MWPR} we can now lift the idempotents from \cref{corollary: idempotents MW and MWP} with the exact same argument as in \cite[Proposition 6.1.4]{quinnZhu2026multiplicativeequivariantthomspectra}.
\end{proof}

\begin{corollary} \label{corollary: idempotents MWR and MWPR}
    There is an $\E_{2\rho}$- resp.~$\E_{3\rho}$-map 
    \[ \MWR \longrightarrow \MWR \quad \text{resp.} \quad \MWPR \longrightarrow \MWPR \] 
    which are idempotents underlying and split off $\MUR$ resp.~$\MUPR$ as an $\E_{2\rho}$- resp.~$\E_{3\rho}$-retract.
\end{corollary}

\begin{proof}
    We apply the previous proposition (\cref{proposition: decomposition of MWR and MWPR}) and the fact that algebras in $\Sc_{C_2}$ are (multiplicatively) augmented, so we construct the idempotents as
    \begin{align*}
        \MWR &\simeq \MUR \otimes \Sigma_+^{\infty}\fib(\Omega^{\infty}\Sigma^{\rho} \gamma_{\R}) \to \MUR \otimes \Sigma_+^{\infty}* \to \MUR \otimes \Sigma_+^{\infty}\fib(\Omega^{\infty}\Sigma^{\rho} \gamma_{\R}) \simeq \MWR
        \\ \MWPR &\simeq \MUPR \otimes \Sigma_+^{\infty}\fib(\Omega^{\infty}\gamma_{\R}) \to \MUPR \otimes \Sigma_+^{\infty}* \to \MUPR \otimes \Sigma_+^{\infty}\fib(\Omega^{\infty}\gamma_{\R}) \simeq \MWPR
    \end{align*}
    by augmentation followed by the unit.
\end{proof}

\begin{theorem} \label{theorem: MUR MUPR lifting result}
    Let $E$ be an $\E_{\infty}^{C_2}$-ring spectrum.
    \begin{enumerate}[(i)]
        \item If $E$ is strongly even, then any $\E_{2n}$-map $\MU \to \Res_e^{C_2} E$ lifts uniquely to an $\E_{n\rho}$-map $\MUR \to E$ for $n \leq 2$.
        \item If $E$ is strongly even periodic, then any $\E_{2n}$-map $\MUP \to \Res_e^{C_2} E$ lifts uniquely to an $\E_{n\rho}$-map $\MUPR \to E$ for $n \leq 3$.
    \end{enumerate}
\end{theorem}

\begin{proof}
    Given the $\E_{\infty}^{C_2}$-orientations from \cref{theorem: commutative maps from MWR and MWPR} and the $\E_{2\rho}$- and $\E_{3\rho}$-splittings from \cref{corollary: idempotents MWR and MWPR} we can run the same argument as in \cite[Theorem 6.2.1]{quinnZhu2026multiplicativeequivariantthomspectra}
\end{proof}

\begin{remark} \label{remark: lower Erho okay}
    In particular, any strongly even $\E_{\infty}^{C_2}$-ring spectrum admits an $\E_{2\rho}$-$\MU_{\R}$-orientation and any strongly even periodic $\E_{\infty}^{C_2}$-ring spectrum admits an $\E_{3\rho}$-$\MU_{\R}$-orientation using \cref{theorem: E4 or E6 orientations}.
\end{remark}

This allows us to improve all the consequences that we obtained in \cite{quinnZhu2026multiplicativeequivariantthomspectra}. However, we need to exercise care: The lifting in \cite{quinnZhu2026multiplicativeequivariantthomspectra} was immediate to do: Let $E$ be a strongly even $\E_{\infty}^{C_2}$-ring spectrum. Using a result from Chadwick--Mandell \cite[Theorem 1.2]{chadwickmandell} we may lift any complex orientation of $\Res_e^{C_2}E$ to an $\E_2$-orientation of $\Res_e^{C_2}E$, which we may then lift to an $\E_{\rho}$-orientation of $E$ by \cref{remark: lower Erho okay}, see \cite[Corollary 6.2.2]{quinnZhu2026multiplicativeequivariantthomspectra}. However, Chadwick--Mandell's result does not generalize to $\E_4$-orientations \cite[Theorem 1.4]{chadwickmandell}, so one needs to input additional (typically hard) theorems to obtain an $\E_4$-orientation on the underlying.

\begin{example} \label{example: E2rho maps}
    In \cite[Corollary 6.2.3, Corollary 6.2.6]{quinnZhu2026multiplicativeequivariantthomspectra} we obtained $\E_{\rho}$-versions of all of the following. We can lift the $\E_4$- and the $\E_6$-orientations from \cref{theorem: E4 or E6 orientations}, but there are also other explicit orientations from the literature that we can lift.
    \begin{enumerate}[(i)]
        \item Let $k$ be a perfect field of characteristic $2$ and $\Gamma_h$ be a formal group law of height $h$ over $k$. Hahn--Shi constructed Real orientations of Lubin--Tate theories $\MUR \to E(k, \Gamma_h)$ in \cite[Theorem 1.1]{hahnRealOrientationsLubin2020}, connecting the obstruction-theoretic $C_2$-action on $E(k, \Gamma_h)$ with the geometry of complex conjugation. In particular, they showed that $E(k, \Gamma_h)$ is strongly even \cite[Theorem 1.9]{hahnRealOrientationsLubin2020}.

        In their work on the telescope conjecture, Burklund--Hahn--Levy--Schlank constructed certain $\E_4$-orientations on $E(\overline{\F}_p,\Gamma_h)$, see \cite[Corollary 5.12]{BHLS}.
        Moreover, in his work on elliptic genera, Senger also obtained $\E_{\infty}$-$\MU$-orientations of height \(\leq 2\) Lubin--Tate theories $E(k,\Gamma_h)$ for \(k\subseteq \overline{\F}_p\) a field of characteristic \(p\) which is algebraic over $\F_p$, see \cite[Corollary 1.5]{sengerleveln}.
        
        Moving to $\MUP$-orientations, Balderrama has constructed $\E_{\infty}$-\(\MUP\)-orientations for height \(\leq 2\) Lubin--Tate theories \cite[Theorem 6.5.3]{balderrama2023algebraictheories}. 
        At long last, via the chromatic Nullstellensatz \cite[Corollary 8.13]{burklund2022chromaticnullstellensatz}, one may obtain an $\E_{\infty}$-orientation $\MUP \to E(k,\Gamma_h)$ for every height $h$ when \(k\) is algebraically closed. 
        
        By \cref{theorem: MUR MUPR lifting result} we may lift these to $\E_{2\rho}$-orientations $\MUR \to E(k, \Gamma_h)$ resp.~to $\E_{3\rho}$-orientations $\MUPR \to E(k, \Gamma_h)$.

        \item Let $n \geq 0$. Meier refined the Hirzebruch level-$n$-genera to $C_2$-equivariant maps $\MUR \to \tmf_1(n)$ and particularly proved that $\tmf_1(n)$ is strongly even \cite[Theorem 2.22]{MeierHirzebruch}. Senger lifted the Hirzebruch level-$n$ genus to an $\E_{\infty}$-map $\MU \to \tmf_1(n)$, so we can lift this to an $\E_{2\rho}$-map $\MUR \to \tmf_1(n)$ by \cref{theorem: MUR MUPR lifting result}.

        This is a first approximation to Senger's question \cite[Question 1.10]{sengerleveln}, about whether there is a $C_2$-equivariant analog for his $\E_{\infty}$-orientations. We will return to fully answering Senger's question in the future.
        \item Burklund--Hahn--Levy--Schlank used the Adams conjecture to consider certain $\E_{\infty}$-algebra maps $\Psi^{\ell} \colon \MU_{(2)} \to \MU_{(2)}$, which they call \emph{Adams operations} on $\MU_{(2)}$, see \cite[Construction 5.3]{BHLS}. So \cref{theorem: MUR MUPR lifting result} lifts these to $\E_{2\rho}$-maps $\MU_{\R(2)} \to \MU_{\R(2)}$.\footnote{Note here that our results also work for $(2)$-localized versions.}
    \end{enumerate}
\end{example}

\begin{theorem} \label{theorem: BPR E2rho}
    The Real Brown--Peterson spectrum $\BPR$ admits an $\E_{2\rho}$-algebra structure.
\end{theorem}

\begin{proof}
    Basterra--Mandell constructed an $\E_4$-idempotent $\MU \to \MU$ which splits off a copy of $\BP$, see \cite[Theorem 1.1]{basterraMandell2013BP}. By \cref{theorem: MUR MUPR lifting result} we may lift this to an $\E_{2\rho}$-algebra map $e \colon \MUR \to \MUR$. We form the filtered colimit
    \[ \BP_{\R}^{\BS} \coloneqq \colim \left( \MUR \xrightarrow{e} \MUR \xrightarrow{e} \cdots \right) \in \Alg_{\E_{2{\rho}}}(\ul{\Sp}^{C_2}) \]
    in $\E_{2\rho}$-algebras. There are now several ways to see that its underlying $C_2$-spectrum is $\BPR$.

    As a filtered colimit of strongly even $C_2$-spectra, $\BP_{\R}^{\BS}$ is strongly even again. Moreover, its underlying spectrum is $\BP$, which is Landweber exact. Thus, it is Real Landweber exact by \cite[Theorem 3.6(b)]{hillmeier2017}. There is only one strongly even $C_2$-spectrum which is Real Landweber exact and underlying $\BP$ by \cite[Proposition 3.13]{hillmeier2017} and $\BPR$ satisfies these conditions. Thus, $\BP_{\R}^{\BS}$ is an $\E_{2\rho}$-refinement of $\BPR$.
\end{proof}

By formal means \cite[Construction 2.2.5, Corollary 6.1.6]{quinnZhu2026multiplicativeequivariantthomspectra} we obtain higher group versions of the previous results.

\begin{remark} \label{remark: chqs}
    In \cite{quinnZhu2026multiplicativeequivariantthomspectra} we showed that $\BPR$ admits an $\E_{\rho}$-algebra structure. Moreover, the techniques there showed that it suffices to compute the $\ul{\Z}$-valued Bredon cohomology of $\BUR\langle 3\rho \rangle \coloneqq \Omega^{\infty}\Sigma^{3\rho}\ku_{\R}$, see \cite[Theorem 5.4.3]{quinnZhu2026multiplicativeequivariantthomspectra}. This computation was carried out by Carrick--Hill--Stewart and the authors and will appear in forthcoming work. So \cref{theorem: BPR E2rho} is an independent second proof of this result.
\end{remark}

\begin{corollary} \label{corollary: higher group structure examples}
    Let $G \geq C_2$ be a finite group.
    \begin{enumerate}[(i)]
        \item Let $G \leq \mathbb{G}_h$ be a finite subgroup of the extended Morava stabilizer group. There exist $\Coind_{C_2}^G \E_{3\rho}$-algebra maps $\MUP^{( \! ( G ) \! )} \to E(k, \Gamma_h)$, where $\Gamma_h$ is a formal group law of height $h$ over an algebraically closed field $k$.
        \item Let $n \geq 0$ and $G = (\Z/n)^{\times}$. Then, the Hirzebruch level-$n$-genera induce $\Coind_{C_2}^G \E_{2\rho}$-algebra maps $\MU^{( \! ( G ) \! )} \to \tmf_1(n)$.
        \item The $G$-spectrum $\BP^{( \! ( G ) \! )}$ admits a $\Coind_{C_2}^G \E_{2\rho}$-algebra structure.
    \end{enumerate}
\end{corollary}
Such orientations of Lubin--Tate theory are of great interest to be able to understand the action of the Morava stabilizer group. A further task is to orient it from smaller objects such as $\BPGm$. We carry this out in forthcoming work of ours \cite{quinnZhu20206structuredquotients} building on work of Beaudry--Hill--Shi--Zeng \cite{beaudryHillShiZeng2021modelsLubinTate}.

\section{Real Snaith equivalences}
\label{section: Real Snaith Equivalences}

\subsection{Real Snaith theorems}

Snaith proved that the important chromatic spectra $\KU$ and $\MUP$ can be recovered as the localizations 
\[ \KU \simeq \Sigma_+^{\infty} \BU(1)[\beta^{-1}] \quad \text{and} \quad \MUP \simeq \Sigma_+^{\infty} \BU [\beta^{-1}] \]
for a certain Bott class $\beta$. While the first equivalence can be refined to one of $\E_{\infty}$-ring spectra, the second one turned out not to enjoy a similar property. Hahn--Yuan proved that it does not refine to an equivalence of $\E_5$-ring spectra, and they at least provide an $\E_2$-equivalence \cite{hahnYuan2020exotic}. Using Wilson space machinery we already produced a structured map  $\MUP \to \Sigma_+^{\infty} \BU[\beta^{-1}]$, more precisely an $\E_6$-map, see \cref{remark: orientation of MUP Snaith}. 

Motivated by this observation and our lifting results \cite{quinnZhu2026multiplicativeequivariantthomspectra}, we will lift Snaith's equivalences to $C_2$-equivariant homotopy theory. 

\begin{construction}
	\hfill \label{construction: bott elements} 
	\begin{enumerate}[(i)]
		\item The inclusion in the second summand
		\[ \phi \colon \Sigma^{\infty} \CP^{\infty}_{\R} \longrightarrow \S \oplus \Sigma^{\infty} \CP^{\infty}_{\R} \simeq \Sigma_+^{\infty} \CP^{\infty}_{\R}. \]
		restricts along the bottom cell inclusion to a map $\Sigma^{\infty} S^{\rho} \simeq \Sigma^{\infty} \CPR^1 \to \Sigma_+^{\infty} \CPR^{\infty}$. This yields the \tb{Real Bott element} $\tb{\beta_{\R}} \in \pi_{\rho}(\Sigma_+^{\infty} \CPR^{\infty})$.
		\item Postcomposing by the preferred map $\Sigma_+^{\infty} \BU_{\R}(1) \to \Sigma_+^{\infty} \BU_{\R}$ yields $\tb{\beta_{\R}} \in \pi_{\rho}(\Sigma_+^{\infty} \BU_\R)$.
	\end{enumerate}
\end{construction}

\begin{remark}
    Let $X$ be a pointed $G$-space. We should note that the map $\Sigma^{\infty}X \to \Sigma_+^{\infty}X$ is not induced by the inclusion map $X \to X_+$ since this is not a map of pointed $G$-spaces. Instead, consider the split fiber sequence
    \begin{center}
        \begin{tikzcd}
            \S \arrow[r, "i"] & \Sigma^{\infty}_+ X \arrow[r, "p"] & \Sigma^{\infty} X.
        \end{tikzcd}
    \end{center}
    Let $r \colon \Sigma_+^{\infty}X \to \S$ denote the right-inverse of $i$. Then, there is a factorization
    \begin{center}
        \begin{tikzcd}
            \S \arrow[r, "i"] & \Sigma_+^{\infty} X \arrow[dr, "\id_{\Sigma_+^{\infty}X}-i \circ r", swap] \arrow[r] & \Sigma^{\infty} X \arrow[d, dashed, "\exists !"]
            \\ & & \Sigma_+^{\infty} X
        \end{tikzcd}
    \end{center}
    by the universal property of cofibers. The map $\phi$ from \cref{construction: bott elements} is an example of this.
\end{remark}

\noindent By construction, the Real Bott elements $\beta_{\R}$ restrict to the classical non-equivariant Bott elements $\beta \in \pi_2(\Sigma_+^{\infty}\CP^{\infty})$ and $\beta \in \pi_2(\Sigma_+^{\infty} \BU)$. Inverting those classes yields Real candidates for Real Snaith equivalences: $\Sigma_+^{\infty} \CP_{\R}^{\infty}[\beta_{\R}^{-1}]$ and $\Sigma_+^{\infty} \BUR[\beta_{\R}^{-1}]$. These are $\E_{\infty}$-algebras in $C_2$-spectra, and we will later be able to show that they admit $\E_{\infty}^{C_2}$-algebra structures (\cref{prop: localizations are commutative}).

\begin{lemma} \label{lemma: Real orientable}
	The $C_2$-spectra $\Sigma_+^{\infty} \CPR^{\infty}[\beta^{-1}_{\R}]$ and $\Sigma_+^{\infty} \BU_{\R}[\beta^{-1}_{\R}]$ are Real orientable.
\end{lemma}

\begin{proof}
	Let us only discuss $\Sigma_+^{\infty} \BU_{\R}[\beta_{\R}^{-1}]$, the same argument works for $\Sigma_+^{\infty} \CP_{\R}^{\infty}[\beta^{-1}_{\R}]$. For this, we observe that the composite
	\begin{center}
		\begin{tikzcd}
			\Sigma^{\infty} \CPR^{\infty} \arrow[r, "\phi"] & \Sigma_+^{\infty} \CPR^{\infty} \arrow[r] & \Sigma_+^{\infty} \BUR[\beta_{\R}^{-1}] \arrow[r, "\beta^{-1}_{\R}"] & \Sigma^{\rho} \Sigma_+^{\infty} \BUR[\beta_{\R}^{-1}],
		\end{tikzcd}
	\end{center}
	is a Real orientation of $\Sigma_+^{\infty} \BUR[\beta^{-1}_{\R}]$ by construction. Indeed, precomposing by $\Sigma^{\infty}\CP_{\R}^1$ yields the following diagram:
    \begin{center}
		\begin{tikzcd}
			\Sigma^{\infty} \CP_{\R}^1 \arrow[r] \arrow[rrr, bend right, "\beta_{\R}", swap]& \Sigma^{\infty} \CPR^{\infty} \arrow[r, "\phi"] & \Sigma_+^{\infty} \CPR^{\infty} \arrow[r] & \Sigma_+^{\infty} \BUR[\beta_{\R}^{-1}] \arrow[r, "\beta^{-1}_{\R}"] & \Sigma^{\rho} \Sigma_+^{\infty} \BUR[\beta_{\R}^{-1}],
		\end{tikzcd}
	\end{center}
    By construction, the first three maps compose to $\beta_{\R}$, so the composition is $\beta_{\R}^{-1} \beta_{\R} = 1$.
\end{proof}

This sets us up to prove the following stronger result. The proof strategy is inspired by \cite[Theorem 3.3]{chathamHahnYuan2024wilson}.

\begin{proposition} \label{prop: strongly even}
	The $C_2$-spectra $\Sigma_+^{\infty} \CPR^{\infty}[\beta^{-1}_{\R}]$ and $\Sigma_+^{\infty} \BU_{\R}[\beta^{-1}_{\R}]$ are strongly even periodic.
\end{proposition}

\begin{proof}
	Let us demonstrate a proof for $\Sigma_+^{\infty} \BUR[\beta_{\R}^{-1}]$, it goes through analogously for $\Sigma_+^{\infty} \CP^{\infty}_{\R}[\beta_{\R}^{-1}]$. It suffices to prove that these are strongly even, the periodicity comes from the fact that $\beta_{\R}$ is an invertible class in degree $\rho$. 
    
    Real orientability (\cref{lemma: Real orientable}) yields a map $f \colon \MU_{\R} \to \Sigma_+^{\infty} \BUR[\beta_{\R}^{-1}]$ of homotopy commutative $C_2$-ring spectra \cite{HuKrizReal}. Then, the composite
	\begin{center}
		\begin{tikzcd}
			\S \otimes \Sigma_+^{\infty} \BUR{[\beta_{\R}^{-1}]} \arrow[r, "1 \otimes \id"] & \MU_{\R} \otimes \Sigma_+^{\infty} \BUR{[\beta_{\R}^{-1}]} \arrow[r, "f \otimes \id"] & {\Sigma_+^{\infty} \BUR[\beta_{\R}^{-1}] \otimes \Sigma_+^{\infty} \BUR[\beta_{\R}^{-1}]} \arrow[d, "\mathrm{mult}"] 
			\\ & & \Sigma_+^{\infty} \BUR{[\beta_{\R}^{-1}]}
		\end{tikzcd}
	\end{center}
	exhibits $\Sigma_+^{\infty} \BUR{[\beta_{\R}^{-1}]}$ as a retract of $\MU_{\R} \otimes \Sigma_+^{\infty} \BUR{[\beta_{\R}^{-1}]}$. On the other hand, $\MU_{\R} \otimes \Sigma_+^{\infty} \BUR$ is strongly even by Real orientation theory \cite[Theorem 4.28]{hill2022freeness}. 
    
    Greenlees' gap characterization \cite[Lemma 1.2]{greenleesfour} describes strongly even $C_2$-spectra as those $C_2$-spectra $X$ with $\pi_{*\rho-i}^{C_2}X = 0$ for $i = 1,2,3$. Since the Real Bott element is in degree $\rho$, the gap condition remains after localizing at the Real Bott element.
\end{proof}

In the presence of equivariant multiplicative structures, norms will show up. On the level of homotopy groups two different (but related) constructions can be meant by this, so we explain our notation to be able to take particular care about this.

\begin{notation} \label{notation: different norms}
    Let $R$ be an $\E_{\infty}^{C_2}$-ring and $n \in \Z$. Suppose that $y \in \pi_n^{e}(R)$.
    \begin{enumerate}[(i)]
        \item We write $\tb{N_e^{C_2}y} \in \pi_{n\rho}^{C_2}(N_e^{C_2}R)$ for the element that arises after applying the functor $N_e^{C_2}$ to the map $y \colon \S^n \to R$.
        \item We write $\tb{\Nm_e^{C_2}y} \in \pi_{n\rho}^{C_2}(R)$ for the element that arises after applying the norm multiplication $N_e^{C_2} \Res_e^{C_2} R \to R$ to $N_e^{C_2}y \in \pi_{n\rho}^{C_2}(N_e^{C_2}R)$.
    \end{enumerate}
\end{notation}

\begin{lemma} \label{lemma: res n res}
    Let $R$ be an $\E_{\infty}^{C_2}$-ring spectrum, $V$ be a $C_2$-representation and $x \in \pi_V^{C_2}(R)$. Let $\deg{\tau}$ denote the degree of the $C_2$-action $\tau$ on $S^V$. Then, 
    \[ \Res_e^{C_2} \Nm_e^{C_2} \Res_e^{C_2} x = (\deg{\tau})(\Res_e^{C_2} x)^2 \in \pi_{2|V|}^e(R). \]
\end{lemma}

\begin{proof}
    Let $\tau \in C_2$ denote the non-identity element. By the norm-restriction double coset formula\footnote{In general for $H \leq G$ it says $\Res_H^G \Nm_H^G a = \prod_{g \in H \setminus G / H} \Nm_{H^g \cap H}^H g_* \res_{H \cap H^g}^H a$. This formula vastly simplifies for $e \leq C_2$.} we compute
    \[ \Res_e^{C_2} \Nm_e^{C_2}(\Res_e^{C_2}x) = (\Res_e^{C_2}x) \cdot \tau_* (\Res_e^{C_2}x) \in \pi_{2|V|}^e(R). \]
    So we have to show $\tau_*(\Res_e^{C_2}x) = \Res_e^{C_2}x$. Because $\Res_e^{C_2}x \colon \S^{|V|} \to \Res_e^{C_2}R$ is restricted from the $C_2$-equivariant map $x$, there is a commutative diagram
    \begin{center}
        \begin{tikzcd}
            \S^{|V|} \arrow[r, "\Res_e^{C_2}x"]  \arrow[d, "\tau", swap] & R \arrow[d, "\tau"]
            \\ \S^{|V|} \arrow[r, "\Res_e^{C_2}x", swap] & R
        \end{tikzcd}
    \end{center}
    This diagram encompasses $\tau_*(\Res_e^{C_2}x) = \tau \circ \Res_e^{C_2}x = \Res_e^{C_2}x \circ \tau$, which on homotopy groups is $(\deg{\tau}) \Res_e^{C_2} x$.
\end{proof}

\begin{corollary} \label{corollary: Nm Res in strongly even}
    Let $R$ be a strongly even $\E_{\infty}^{C_2}$-ring spectrum and $n \in \Z$. Suppose that $x \in \pi_{n \rho}^{C_2}(R)$. Then, $\Nm_e^{C_2} \Res_e^{C_2} x = (-1)^n x^2 \in \pi_{2n\rho}^{C_2}(R)$.
\end{corollary}

\begin{proof}
    Combine the previous lemma (\cref{lemma: res n res}) with the strongly even property, i.e.~that the restriction map $\Res_e^{C_2} \colon \pi_{2n\rho}^{C_2}(R) \to \pi_{4n}^e(R)$ is an isomorphism.
\end{proof}

\begin{proposition} \label{prop: localizations are commutative}
	The $C_2$-spectra $\Sigma_+^{\infty} \CPR^{\infty}[\beta^{-1}_{\R}]$ and $\Sigma_+^{\infty} \BU_{\R}[\beta^{-1}_{\R}]$ admit $\E_{\infty}^{C_2}$-algebra structures.
\end{proposition}

\begin{proof}
    Let us only discuss $\Sigma_+^{\infty} \BUR[\beta_{\R}^{-1}]$, the same argument works for $\Sigma_+^{\infty} \CP^{\infty}_{\R}[\beta^{-1}_{\R}]$. Let us first note that $\Sigma_+^{\infty} \BUR$ admits an $\E_{\infty}^{C_2}$-algebra structure since $\BUR \simeq \Omega^{\infty} \Sigma^{\rho} \ku_{\R}$ does. Now, let $\ul{S} \subseteq \ul{\pi}_{\star}(\Sigma_+^{\infty}\BUR)$ be the smallest $C_2$-subset of the $\RO(C_2)$-graded $C_2$-homotopy groups of $\Sigma_+^{\infty} \BUR$ that is closed under norms and contains $\beta_{\R} \in \pi_{\rho}^{C_2}(\Sigma_+^{\infty} \BUR)$.
    
    Then, the module localization $\Sigma_+^{\infty}\BUR[\ul{S}^{-1}]$ admits an $\E_{\infty}^{C_2}$-algebra structure by \cite[Proposition 4.9]{hillHopkins2014equivariantmultiplicativeclosure}. We wish to see 
    \[ \Sigma_+^{\infty} \BUR[\ul{S}^{-1}] \simeq \Sigma_+^{\infty} \BUR[\beta_{\R}^{-1}][\ul{S}^{-1}] \simeq \Sigma_+^{\infty} \BUR[\beta_{\R}^{-1}]. \] 
    First, let us write $\beta \coloneqq \Res_e^{C_2} \beta_{\R} \in \pi_2^e(\Sigma_+^{\infty} \BUR[\beta_{\R}^{-1}]) \cong \pi_2(\Sigma_+^{\infty} \BUR[\beta^{-1}])$. Next, let us consider the element $N_e^{C_2} \beta \in \pi_{2\rho}(\Sigma_+^{\infty} \BUR[\beta_{\R}^{-1}])$. Since $\Sigma_+^{\infty} \BUR[\beta_{\R}^{-1}]$ is strongly even (\cref{prop: strongly even}), we deduce $\Nm_e^{C_2}\beta = \beta_{\R}^2$ by \cref{corollary: Nm Res in strongly even}. 
    
    With this computation we now see that the $C_2$-set $\ul{S}$ is given by
    \[ \ul{S}_{C_2} = \{\beta_{\R}^i \Nm_e^{C_2}(\beta)^j \mid i,j \geq 0 \} \quad \text{and} \quad \ul{S}_e = \{\beta^k \mid k \geq 0 \}. \]
    All of these elements are already inverted in $\Sigma_+^{\infty} \BUR[\beta_{\R}^{-1}]$.
\end{proof}

\begin{theorem} \label{theorem: Erho Real Snaith}
	There are equivalences
    \[ \KU_{\R} \simeq \Sigma_+^{\infty} \CPR^{\infty}[\beta_{\R}^{-1}] \quad \text{and} \quad \MUP_{\R} \simeq \Sigma_+^{\infty} \BU_{\R}[\beta_{\R}^{-1}] \]
	of $\E_{\infty}^{C_2}$-algebras resp.~of $\E_{\rho}$-algebras.
\end{theorem}

\begin{proof}
	Since equivalences between strongly even $C_2$-spectra are checked underlying \cite[Lemma 3.4]{hillmeier2017}, it suffices to lift the classical Snaith equivalences to $C_2$-equivariant ones. 
    
    Let us begin with the $\KU_{\R}$-Snaith. By the universal property of localization \cite[Proposition 4.9]{hillHopkins2014equivariantmultiplicativeclosure} and the computation of the multiplicative equivariant closure of $\beta_{\R}$ in \cref{prop: localizations are commutative}, it suffices to write down an $\E_{\infty}^{C_2}$-algebra map $\Sigma_+^{\infty} \CPR^{\infty} \to \KU_{\R}$, which inverts $\beta_{\R}$ and refines the classical Snaith map. Such a map can be written down in models, as e.g.~done in \cite[Construction 4.3]{schwede2026realglobalequivariantsegalbeckersplitting}. Since the classical Snaith map inverts $\beta$, the lift must invert $\beta_{\R}$ by strong evenness.

    Onto the $\MUP_{\R}$-Snaith. By \cref{prop: strongly even} and \cref{prop: localizations are commutative} the assumptions in \cref{theorem: MUR MUPR lifting result} are satisfied, so any $\E_2$-Snaith equivalence $\MUP \to \Sigma_+^{\infty}\BU[\beta^{-1}]$ provided by Hahn--Yuan \cite[Proposition 3.1]{hahnYuan2020exotic} lifts uniquely to an $\E_{\rho}$-algebra map $\MUPR \to \Sigma_+^{\infty}\BUR[\beta_{\R}^{-1}]$.
\end{proof}

\begin{remark}
    Our proof in particular shows that every (structured) equivalence $\MUP \to \Sigma_+^{\infty}\BU[\beta^{-1}]$ admits a unique lift to the Real setting. Similarly, using \cite[Theorem D]{quinnZhu2026multiplicativeequivariantthomspectra}, every equivalence $\Sigma_+^{\infty}\CP^{\infty}[\beta^{-1}] \to \KU$ of spectra lifts uniquely to an equivalence $\Sigma_+^{\infty} \CP_{\R}^{\infty}[\beta_{\R}^{-1}] \to \KU_{\R}$ of $C_2$-spectra. In fact, this strategy gives an alternative proof of  $\KU_{\R}$-Snaith by lifting the non-equivariant Snaith map instead of writing down the map using models. The downside is that we are currently not able to obtain lifts of structured maps, so we only recover a non-structured $\KU_{\R}$-Snaith.
\end{remark}

\begin{remark}
    The Real Snaith theorem shows that $\Sigma_+^{\infty} \CP_{\R}^{\infty}[\beta_{\R}^{-1}]$ and $\Sigma_+^{\infty} \BUR[\beta_{\R}^{-1}]$ are Borel $C_2$-spectra.
\end{remark}

Let $G \geq C_2$ be a finite group. Applying $N_{C_2}^G$ yields higher group versions of the Real Snaith theorem, yielding a variant of \cite[Theorem 3.33]{hill2022freeness}.

\begin{corollary}
    Let $G \geq C_2$ be a finite group. There are equivalences
    \[ N_{C_2}^G \KU_{\R} \simeq \Sigma_+^{\infty} \Map_{C_2}(G, \CP_{\R}^{\infty})[N_{C_2}^G(\beta_{\R})^{-1}] \quad \text{and} \quad \MUP^{( \! ( G ) \! )} \simeq \Sigma_+^{\infty} \Map_{C_2}(G, \BUR)[N_{C_2}^G(\beta_{\R})^{-1}] \]
    of $\E_{\infty}^G$-algebras resp.~of $\Coind_{C_2}^G \E_{\rho}$-algebras.
\end{corollary}

\begin{remark} \label{remark: history etc on snaith}
    \hfill 
    \begin{enumerate}[(i)]
        \item Gepner--Snaith \cite{gepnerSnaith2009motivicspectrarepresentingalgebraic}\footnote{Beware that there is a mistake in this paper, see \cite[Footnote 9]{annalaHoyoisIwasa2025algebraiccobordismconnerfloydisomorphism}.} resp.~Annala--Iwasa \cite[Proposition 4.3.3]{annalaIwasa2025motivicspectrauniversalityktheory} and Annala--Hoyois--Iwasa \cite[Theorem 9.3]{annalaHoyoisIwasa2025algebraiccobordismconnerfloydisomorphism} provided motivic analogues of the Snaith equivalences, so the Real Snaith theorems without any equivariant multiplicative structures were already implicitly known by $C_2$-Betti realization. Schwede had furthermore already generalized the $\KU$-Snaith to an equivalence of ultracommutative Real global spectra before us, thereby also providing a much more structured result. It will appear in forthcoming work of his, as e.g.~was announced in \cite[Construction 4.3]{schwede2026realglobalequivariantsegalbeckersplitting}. His student Song is currently working on a Real global analog of $\MUP$-Snaith.

        Nonetheless, our strategy via strong evenness provides a shortcut in the (non-global) Real setting. In particular, the $\E_{\rho}$-structure of the $\MUPR$-equivalence is novel and demonstrates the versatility of our lifting results \cite{quinnZhu2026multiplicativeequivariantthomspectra}.
        \item If $\Sigma_+^{\infty}\BU[\beta^{-1}] \simeq \MUP$ lifts to an $\E_4$-equivalence as conjectured by Hahn--Yuan \cite[Question 3]{hahnYuan2020exotic}, then our technique would allow us to lift this to an $\E_{2\rho}$-equivalence.
    \end{enumerate}
\end{remark}

Applying flavours of $C_2$-fixed points to the Real Snaith theorem yields real Snaith theorems. Let us spell out these results. Taking $C_2$-geometric fixed points turns the Real Bott element (\cref{construction: bott elements}) into the Hopf element $\eta \colon S^1 \to \RP^{\infty}$ resp.~$\eta \colon S^1 \to \RP^{\infty} \to \BO$.

The following is certainly known classically, but we use this opportunity to demonstrate an application of the Real Snaith theorem. It will afterwards appear in our description of the real Snaith theorems.

\begin{lemma} \label{lemma: RP bott inverted is 0}
    There is an equivalence $\Sigma_+^{\infty} \RP^{\infty}[\eta^{-1}] \simeq 0$.
\end{lemma}

\begin{proof}
    Applying $\Phi^{C_2}$ to the $\KU_{\R}$-Snaith (\cref{theorem: Erho Real Snaith}) yields $\Sigma_+^{\infty} \RP^{\infty}[\eta^{-1}] \simeq \Phi^{C_2} \KU_{\R} \simeq 0$.
\end{proof}

\begin{corollary} \label{corollary: real Snaith}
	There are equivalences 
	\[ \KO \simeq \Sigma_+^{\infty} \CP^{\infty}_{\R}[\beta_{\R}^{-1}]_{hC_2} \quad \text{and} \quad \MOP \simeq \Sigma_+^{\infty} \BO[\eta^{-1}]. \]
	The latter enhances to an $\E_1$-equivalence.
\end{corollary}

\begin{proof}
	The second equivalence follows from the Real Snaith theorem (\cref{theorem: Erho Real Snaith}) after applying geometric fixed points $\Phi^{C_2}$, and the $\E_1$-structure comes from the $\E_{\rho}$-structure. For the first equivalence, applying $(-)^{C_2}$ to the Real Snaith theorem (\cref{theorem: Erho Real Snaith}) and tom Dieck splitting first yields
    \[ \KO \simeq \Sigma_+^{\infty} \RP^{\infty}[\eta^{-1}] \oplus \Sigma_+^{\infty} \CP^{\infty}_{\R}[\beta_{\R}^{-1}]_{hC_2}. \]
    On the other hand, $\Sigma_+^{\infty} \RP^{\infty}[\eta^{-1}] \simeq 0$ by \cref{lemma: RP bott inverted is 0}.
\end{proof}

\begin{remark}
    Similarly, \cref{theorem: Erho Real Snaith} with tom Dieck splitting gives rise to a formula for $\MUP_{\R}^{C_2}$ as well, namely $\MUP_{\R}^{C_2} \simeq \Sigma_+^{\infty} \BO[\eta^{-1}] \oplus \Sigma_+^{\infty} \BUR[\beta_{\R}^{-1}]_{hC_2}$.
\end{remark}

\subsection{Norm inverted Real Snaith theorems}
In preparation for $\THR$ computations (\cref{theorem: THR KUR}) we will now use the Real Snaith theorems (\cref{theorem: Erho Real Snaith}) to obtain an alternate Snaith construction. Let us begin by recording a (certainly well-known) classical fact about localizations.

\begin{lemma} \label{lemma: localization and nilpotent}
    Let $R$ be an $\E_{\infty}$-ring spectrum and $n \in \N$ with $x,\varepsilon \in \pi_n R$. Suppose that $\varepsilon$ is nilpotent in $\pi_* R$. Then, the localization map $R \to R[(x+\varepsilon)^{-1}]$ induces an equivalence $R[x^{-1}] \xrightarrow{\ \simeq \ } R[(x+\varepsilon)^{-1}]$ of $\E_{\infty}$-ring spectra.
\end{lemma}

\begin{proof}
    First, we need to show that $x$ is invertible in $R[(x+\varepsilon)^{-1}]$ so that the localization map induces a map $R[x^{-1}] \to R[(x+\varepsilon)^{-1}]$. We see this by writing $x = (x+\varepsilon) \left(1-\frac{\varepsilon}{x+\varepsilon} \right)$, which is a product of two units: $x+\varepsilon$ is invertible in $R[(x+\varepsilon)^{-1}]$ by construction and the second term is the sum of a unit and a nilpotent element and thus invertible.

    Conversely, $x+\varepsilon$ as a sum of a unit and a nilpotent element is invertible in $R[x^{-1}]$, so the localization map $R \to R[x^{-1}]$ induces a map $R[(x+\varepsilon)^{-1}] \to R[x^{-1}]$.

    Composing these two maps yields two maps
    \[ R[x^{-1}] \longrightarrow R[x^{-1}] \quad \text{and} \quad R[(x+\varepsilon)^{-1}] \longrightarrow R[(x+\varepsilon)^{-1}], \]
    both induced by the localization map. Thus, these must be the identities, which means that we have constructed inverse maps.
\end{proof}

\begin{lemma} \label{theorem: Nm localization snaith constructions}
    The $C_2$-spectra $\Sigma_+^{\infty} \CP^{\infty}_{\R}[\Nm_e^{C_2}(\beta)^{-1}]$ and $\Sigma_+^{\infty} \BUR[\Nm_e^{C_2}(\beta)^{-1}]$ admit $\E_{\infty}^{C_2}$-algebra structures such that there are equivalences
    \[ \Sigma_+^{\infty} \CP^{\infty}_{\R}[\Nm_e^{C_2}(\beta)^{-1}] \simeq \Sigma_+^{\infty} \CP_{\R}^{\infty}[\beta_{\R}^{-1}] \quad \text{and} \quad \Sigma_+^{\infty} \BUR[\Nm_e^{C_2}(\beta)^{-1}] \simeq \Sigma_+^{\infty} \BUR[\beta_{\R}^{-1}] \]
    of $\E_{\infty}^{C_2}$-algebras.
\end{lemma}

\begin{proof}
    There are maps
    \[ \Sigma_+^{\infty} \CP^{\infty}_{\R}[\Nm_e^{C_2}(\beta)^{-1}] \longrightarrow \Sigma_+^{\infty} \CP_{\R}^{\infty}[\beta_{\R}^{-1}] \quad \text{and} \quad \Sigma_+^{\infty} \BUR[\Nm_e^{C_2}(\beta)^{-1}] \longrightarrow \Sigma_+^{\infty} \BUR[\beta_{\R}^{-1}] \]
    of $\E_{\infty}$-algebras, and it suffices to show that these are equivalences. This then shows in particular, that the sources obtain $\E_{\infty}^{C_2}$-algebra structures by Hill--Hopkins \cite{hillHopkins2014equivariantmultiplicativeclosure}. We do this by checking it on geometric fixed points. Since $\Res_e^{C_2}\Nm_e^{C_2}(\beta) = \beta^2$ by \cref{lemma: res n res}, the map is evidently an equivalence after applying $\Phi^e$. So let us discuss $\Phi^{C_2}$.

    Let us start with the $\CP_{\R}^{\infty}$-version. First, we observe $\Phi^{C_2} \Sigma_+^{\infty}\CP_{\R}^{\infty}[\beta_{\R}^{-1}] \simeq \Phi^{C_2} \KU_{\R} \simeq 0$ by the Real Snaith theorem (\cref{theorem: Erho Real Snaith}). We need to compare it to $\Sigma_+^{\infty}\RP^{\infty}[(r\beta)^{-1}]$ where $r\beta = \Phi^{C_2}\Nm_e^{C_2}(\beta)$ is the element 
    \[ r\beta \colon \S^2 \simeq \Phi^{C_2} N_e^{C_2} \S^2 \xrightarrow{\ \beta \ } \Phi^{C_2} N_e^{C_2} \Sigma_+^{\infty}\CP^{\infty} \xrightarrow{\ r \ } \Phi^{C_2} \Sigma_+^{\infty}\CP^{\infty}_{\R} \simeq \Sigma_+^{\infty}\RP^{\infty}. \]
    In other words, we need to show that $r\beta \in \pi_2(\Sigma_+^{\infty} \RP^{\infty})$ is nilpotent. We achieve this with the Devinatz--Hopkins--Smith nilpotence theorem \cite{devinatzHopkinsSmith1988nilpotence1}.\footnote{One could also have used the May nilpotence theorem.} Indeed, $\pi_2(\Sigma_+^{\infty} \RP^{\infty})$ is torsion since $H_2(\Sigma_+^{\infty}\RP^{\infty};\mathbb Q) = 0$, but $\MU_2(\Sigma_+^{\infty} \RP^{\infty}) \cong \Z$.\footnote{The map $\MU \to \Z$ is $2$-connective, so $\fib(\MU \to \Z) \otimes \Sigma^{\infty} \RP^{\infty}$ is $3$-connective. By the associated LES the map $\MU_2(\Sigma^{\infty}\RP^{\infty}) \to H_2(\Sigma^{\infty}\RP^{\infty}) \cong 0$ is an isomorphism.} This is torsionfree. So $r\beta$ lies in the kernel of the Hurewicz map of $\MU$ and thus is nilpotent by the nilpotence theorem.

    Now onto the $\BU_{\R}$-version. We consider the two elements $r\beta = \Phi^{C_2} \Nm_e^{C_2}(\beta)$ and $\eta^2 = \Phi^{C_2} \beta_{\R}^2$ defined as follows:\footnote{In particular, we abuse notation and use the notation $r\beta$ again as in the $\CP_{\R}^{\infty}$-version.}
    \begin{align*}
        r\beta &\colon \S^2 \simeq \Phi^{C_2} N_e^{C_2} \S^2 \xrightarrow{\ \beta\ } \Phi^{C_2} N_e^{C_2} \Sigma_+^{\infty}\BU_{\R} \xrightarrow{\ r \ } \Phi^{C_2} \Sigma_+^{\infty}\BU_{\R} \simeq \Sigma_+^{\infty}\BO,
        \\ \eta^2 &\colon \S^2 \simeq \Phi^{C_2} \S^{2\rho} \xrightarrow{\ \beta^2 \ } \Phi^{C_2} \Sigma_+^{\infty} \BU_{\R} \simeq \Sigma_+^{\infty} \BO.
    \end{align*}
    Thus, we need to show that the map of $\E_{\infty}$-ring spectra
    \[ \Sigma_+^{\infty} \BO[(r\beta)^{-1}] \longrightarrow \Sigma_+^{\infty}\BO[\eta^{-2}] \] 
    is an equivalence. By \cref{lemma: localization and nilpotent}, it suffices to prove that $r\beta - \eta^2$ is nilpotent. We wish to employ May nilpotence \cite{MNNMayNilpotence}. So consider the Hurewicz homomorphism
    \[ h \colon \pi_2(\Sigma_+^{\infty}\BO) \longrightarrow H_2(\BO; \Z). \]
    A Serre spectral sequence computation combined with the universal coefficient theorem shows that there is an isomorphism $H_2(\BO;\Z) \cong \Z/2$. 
    Since $\Sigma_+^{\infty}\BO[\eta^{-2}] \simeq \MOP \neq 0$ by \cref{corollary: real Snaith}, the element $\eta^2$ cannot be nilpotent, so $h(\eta^2) = 1$ by May nilpotence. On the other hand, we have a map of $\E_{\infty}$-ring spectra $\Sigma_+^{\infty}\BO[(r\beta)^{-1}] \to \Sigma_+^{\infty}\BO[\eta^{-2}]$ with non-zero target, so the source must also be non-zero. Hence, $r\beta$ is not nilpotent, whence $h(r\beta) = 1$ by May nilpotence. By linearity, we conclude 
    \[ h(r\beta - \eta^2) = h(r\beta) - h(\eta^2) = 1 - 1 = 0, \] 
    so $r\beta - \eta^2 \in \pi_*(\Sigma_+^{\infty}\BO)$ is nilpotent by May nilpotence, as desired.
\end{proof}

Combining this with the Real Snaith theorems (\cref{theorem: Erho Real Snaith}) immediately yields:

\begin{theorem} \label{corollary: norm inverted snaith}
    There are equivalences
    \[ \KU_{\R} \simeq \Sigma_+^{\infty} \CP^{\infty}_{\R}[\Nm_e^{C_2}(\beta)^{-1}]  \quad \text{and} \quad \MUP_{\R} \simeq \Sigma_+^{\infty} \BUR[\Nm_e^{C_2}(\beta)^{-1}] \]
    of $\E_{\infty}^{C_2}$- resp.~$\E_{\rho}$-algebras.
\end{theorem}

Taking $(-)^{C_2}$ and applying tom Dieck splitting as in \cref{corollary: real Snaith} yields a norm inverted version of the real Snaith theorem for $\KO$:

\begin{corollary}
    There is an equivalence $\KO \simeq \Sigma_+^{\infty} \CP_{\R}^{\infty}[\Nm_e^{C_2}(\beta)^{-1}]_{hC_2}$ of $\E_{\infty}$-algebras.
\end{corollary}

\section{Computations in Real topological Hochschild homology}
\label{section: computations in THR}

\subsection{Structured Real Bökstedt equivalence \& Inverting elements in THR}
We will now give applications to $\THR$-computations. Recall that for an $\E_{\sigma}$-algebra $R$ one can define $\THR(R) \coloneqq R \otimes_{N_e^{C_2} R^e} R$, which for $\E_{\infty}^{C_2}$-algebras $R$ can be described as $\THR(R) \simeq R \otimes S^{\sigma}$, see \cite{dottoMoiPatchkoriaReeh2021THR} and \cite[Remark 5.3]{quigleyShah2022equivalencetheoriesrealcyclotomic}.

Our first $\THR$-computation concerns Real Bökstedt periodicity. Bayındır--Moulinos showed that Bökstedt periodicity $\THH(\F_p) \simeq \F_p \otimes \Sigma_+^{\infty}\Omega S^3$ refines to an $\E_2$-equivalence \cite[Theorem 1.3]{bayindirMoulinos2022KTHHFp}, and we lift it to an $\E_{\rho}$-equivalence. We include it here to demonstrate that our strategy to prove the Real Snaith theorems is versatile enough to be applicable in multiple settings.

\begin{theorem} \label{theorem: structured Real Boekstedt}
    There is an equivalence $\THR(\ul{\F}_p) \simeq \ul{\F}_p \otimes \Sigma_+^{\infty}\Omega S^{1+\rho}$ of $\E_{\rho}$-algebras.
\end{theorem}

\begin{proof}
    It is known that there is such an equivalence of $C_2$-spectra \cite[Theorem 5.18]{dottoMoiPatchkoriaReeh2021THR}, which in particular implies that it is strongly even. Since $\ul{\F}_p$ is naturally $\E_{\infty}^{C_2}$, it induces an $ \ul{\F}_p$-$\E_{\infty}^{C_2}$-algebra structure on $\THR(\ul{\F}_p)$, see \cref{observation: algebra structure on THR}. Moreover, $S^{1+\rho}$ admits an $\E_{\sigma}$-algebra structure, since there is an equivalence $S^{1+\rho} \simeq \Omega^{\sigma} \HP^{\infty}_{\R}$ by \cite[Proposition 4.2]{hahnRealOrientationsLubin2020},\footnote{The $C_2$-action on $\HP^{\infty}$ is given by conjugation by $i$.} so the other side of the equivalence admits an $ \ul{\F}_p$-$\E_{\rho}$-algebra structure as well.
    
    Now, $\ul{\F}_p$-$\E_{\rho}$-maps $\ul{\F}_p \otimes \Sigma_+^{\infty} \Omega S^{1+\rho} \to \THR(\ul{\F}_p)$ correspond to maps 
    \[ \Sigma^{\infty} \HP_{\R}^{\infty} \simeq \Sigma^{\infty} \mathrm{B}^{\sigma} S^{1+\rho} \longrightarrow \gl_1 \THR(\ul{\F}_p) \] 
    of $C_2$-spectra by an adjunction argument, see e.g.~\cite[Proposition 5.4.1]{quinnZhu2026multiplicativeequivariantthomspectra}. On underlying, there is a map which corresponds to an $\E_2$-equivalence $\THH(\F_p) \simeq  \F_p \otimes \Sigma_+^{\infty}\Omega S^3$ by Bayındır--Moulinos \cite[Theorem 1.3]{bayindirMoulinos2022KTHHFp}. It suffices to lift this map since equivalences of strongly even spectra can be checked underlying \cite[Lemma 3.4]{hillmeier2017}. To do this, we want to apply \cite[Theorem 5.3.8]{quinnZhu2026multiplicativeequivariantthomspectra}, for which the only missing assumption is that $\Sigma^{\infty} \HP^{\infty}_{\R}$ has finite type $\rho$-free homology. 
    
    To see this, we consider a $C_2$-cell decomposition of $\HP^{\infty}_{\R}$ suggested by Mike Hopkins and spelled out by Hahn--Shi \cite[Proposition 4.2]{hahnRealOrientationsLubin2020} -- namely one with a single $2n\rho$-cell for each $n \geq 0$. Thus, $\Sigma_+^{\infty} \HP^{\infty}_{\R}$ has finite type $\rho$-free homology analogously to \cite[Section 4.3.1]{hill2022freeness}, see e.g.~\cite[Example 5.2]{pitschRickaScherer2021conjugationspaces}. 
\end{proof}

Next, we discuss the interaction of $\THR$ with inversions. With it we set up the main ingredient in computing $\THR(\KU_{\R})$ and $\THR(\MUPR)$ using the Real Snaith theorems (\cref{theorem: Erho Real Snaith}), see \cref{theorem: THR KUR}.

\begin{observation} \label{observation: algebra structure on THR}
    Let $R$ be an $\E_{\infty}^{C_2}$-ring spectrum. Then, the counit $N_e^{C_2} \Res_e^{C_2} R \to R$ is an $\E_{\infty}^{C_2}$-algebra map. In particular, the map
    \[ R \simeq N_e^{C_2} \Res_e^{C_2} R \otimes_{N_e^{C_2} \Res_e^{C_2} R} R \longrightarrow R \otimes_{N_e^{C_2} \Res_e^{C_2} R} R \simeq \THR(R) \]
    is an $\E_{\infty}^{C_2}$-algebra map, thus endowing $\THR(R)$ with an $R$-module structure. This gives sense to invert homotopy group elements from $R$ in $\THR(R)$ in an $R$-module fashion.
\end{observation}

Both $N_e^{C_2}$ and $\Nm_e^{C_2}$ will show up in the proof of the following statement, so we refer the reader to \cref{notation: different norms} for a reminder of this notation. Our remark ensures that the objects in the statement are well-defined.

\begin{proposition} \label{prop: THR inversion}
    Let $R$ be an $\E_{\infty}^{C_2}$-ring spectrum and $V$ be a $C_2$-representation. Suppose that $\ol{x} \in \pi_V^{C_2}(R)$ and $x = \Res_e^{C_2}\ol{x} \in \pi_{|V|}^e(R)$. Then, there is an equivalence
    \[ \THR \left(R[(\Nm_e^{C_2}x)^{-1}] \right) \simeq \THR(R)[(\Nm_e^{C_2}x)^{-1}] \]
    of $\E_{\infty}^{C_2}$-algebras.
\end{proposition}

\begin{proof}
    Let us first note that both sides admit $\E_{\infty}^{C_2}$-algebra structures by Hill--Hopkins \cite{hillHopkins2014equivariantmultiplicativeclosure}, see also our argument in \cref{prop: localizations are commutative}.
    
    Now, we can compute 
    \begin{align*}
        \THR \left(R[(\Nm_e^{C_2}x)^{-1}] \right) &\simeq R[(\Nm_e^{C_2}x)^{-1}] \otimes_{N_e^{C_2} \Res_e^{C_2} (R[(\Nm_e^{C_2}x)^{-1}])} R[(\Nm_e^{C_2}x)^{-1}]
        \\ &\simeq R[(\Nm_e^{C_2}x)^{-1}] \otimes_{(N_e^{C_2} \Res_e^{C_2} R)[(N_e^{C_2}x)^{-1}]} R[(\Nm_e^{C_2}x)^{-1}].
    \end{align*}
    On the other hand, $R[(\Nm_e^{C_2}x)^{-1}] \simeq (N_e^{C_2} \Res_e^{C_2} R)[(N_e^{C_2}x)^{-1}] \otimes_{(N_e^{C_2} \Res_e^{C_2} R)} R$, so this term simplifies into
    \[ R[(\Nm_e^{C_2}x)^{-1}] \otimes_{(N_e^{C_2} \Res_e^{C_2} R)} R \simeq \THR(R)[(\Nm_e^{C_2}x)^{-1}], \]
    as desired.
\end{proof}

\begin{remark}
    The object 
    \[ (N_e^{C_2} \Res_e^{C_2} R)[(N_e^{C_2}x)^{-1}] \otimes_{(N_e^{C_2} \Res_e^{C_2} R)} R \simeq N_e^{C_2} \Res_e^{C_2} (R[(\Nm_e^{C_2}x)^{-1}]) \otimes_{(N_e^{C_2} \Res_e^{C_2} R)} R \] 
    appearing in the above proof is also known as $N_{R}^{e \to C_2} (R[(\Nm_e^{C_2}x)^{-1}])$, the norm in the $C_2$-symmetric monoidal $\infty$-category $\ul{\LMod}_R$, see e.g.~\cite[Example 3.3.14]{quinnZhu2026multiplicativeequivariantthomspectra}.
\end{remark}

\begin{corollary} \label{corollary: THR inversion strongly even}
    Let $R$ be a strongly even $\E_{\infty}^{C_2}$-ring spectrum and $n \in \Z$. Consider $\ol{x} \in \pi_{n \rho}^{C_2}(R)$. Then, there is an equivalence
    \[ \THR(R[\ol{x}^{-1}]) \simeq \THR(R)[\ol{x}^{-1}] \]
    of $\E_{\infty}^{C_2}$-ring spectra. 
\end{corollary}

\begin{proof}
    By \cref{corollary: Nm Res in strongly even} we have $\Nm_e^{C_2} \Res_e^{C_2} \ol{x} = (-1)^n \ol{x}^2$, but inverting $(-1)^n \ol{x}^2$ is the same thing as inverting $\ol{x}$. Moreover, one can now employ the same argument as in \cref{prop: localizations are commutative} to obtain $\E_{\infty}^{C_2}$-ring structures on the localizations.
\end{proof}

We can recover the classical result about commuting $\THH$ with inverting elements  \cite[Corollary 4.12]{stonek2020thhku}.

\begin{corollary} \label{corollary: THH commutes with inverting}
    Let $R$ be an $\E_{\infty}$-ring spectrum, then we can recover the classical equivalence
    \[ \THH(R[x^{-1}]) \simeq \THH(R)[x^{-1}] \]
    for every $x \in \pi_* R$.
\end{corollary}

\begin{proof}
    Endowed with the trivial action, the Borelification of $R$ becomes an $\E_{\infty}^{C_2}$-ring spectrum by \cite[Theorem 2.4, 2.7]{hillmeier2017}. Now, consider the element $\ol{x} = \Nm_e^{C_2} x$. Then, $\Res_e^{C_2} \ol{x} = x \cdot \tau_* x = x^2$ by the norm-restriction double coset formula remembering that $R$ has trivial $C_2$-action. Since inverting $x^2$ is equivalent to inverting $x$, the underlying of \cref{prop: THR inversion} yields the $\THH$ statement.
\end{proof}

\subsection{Applications of Real Snaith in THR computations}

Our first $\THR$-result applies the Real Snaith theorem. Stonek showed $\THH(\KU) \simeq \KU[\mathrm{B}^2\mathrm{U}(1)]$ using Snaith's theorem, see \cite[Theorem 5.21]{stonek2020thhku} and \cite[Example 7.1]{rasekhStonekValenzuela2022thom}. Using the Real Snaith theorem, we can lift this equivalence.

\begin{theorem} \label{theorem: THR KUR}
	There are equivalences 
    \[ \THR(\KU_{\R}) \simeq \KU_{\R} \otimes \Sigma_+^{\infty} \mathrm{B}^\rho \mathrm{U}_{\R}(1) \quad \text{resp.} \quad \THR \left(\MUP_{\R} \right) \simeq \MUP_{\R} \otimes \Sigma_+^{\infty} \B^{\rho}\mathrm{U}_{\R} \] of $\E_{\infty}^{C_2}$-algebras resp.~of $\E_1$-algebras.
\end{theorem}

\begin{proof}
    Let us demonstrate the computation for $\THR(\MUP_{\R})$, the same works for $\THR(\KU_{\R})$. By the norm inverted Real Snaith theorem (\cref{corollary: norm inverted snaith}) we have $\THR(\MUP_{\R}) \simeq \THR(\Sigma_+^{\infty} \BU_{\R}[\Nm_e^{C_2}(\beta)^{-1}])$. Using \cref{prop: THR inversion} we may commute the localization with $\THR$. Combined with the computation of $\THR$ of spherical group rings \cite[Proposition 5.12]{dottoMoiPatchkoriaReeh2021THR}, we obtain
    \begin{align*}
        \THR(\MUP_{\R}) &\simeq \THR(\Sigma_+^{\infty} \BU_{\R})\left[\Nm_e^{C_2}(\beta)^{-1} \right] 
        \\ &\simeq \left(\Sigma_+^{\infty} \BU_{\R} \otimes \Sigma_+^{\infty} \B^{\rho} \mathrm{U}_{\R} \right)\left[\Nm_e^{C_2}(\beta)^{-1} \right]
        \\ &\simeq \MUP_{\R} \otimes \Sigma_+^{\infty} \B^{\rho} \mathrm{U}_{\R},
    \end{align*}
    as desired.
\end{proof}

Applying $\Phi^{C_2}$ yields:

\begin{corollary}
    There is an $\E_1$-equivalence $\MOP \otimes_{\MUP} \MOP \simeq \MOP \otimes \Sigma_+^{\infty} \B^{\rho}\mathrm{U}_{\R}^{C_2}$.
\end{corollary}

\begin{remark}
    Using techniques to compute $\THR$ of Thom spectra via factorization homology Horev--Klang--Zou compute $\THR(\MU_{\R}) \simeq \MU_{\R} \otimes \Sigma_+^{\infty} \B^{\rho} \mathrm{U}_{\R}$, see \cite[Corollary 7.2.2]{hahn2024equivariantnonabelianpoincareduality}.  This is however conditional on a $G$-symmetric monoidal straightening-unstraightening result \cite[Remark 7.1.2]{quinnZhu2026multiplicativeequivariantthomspectra}, which is not yet recorded in the literature. We will discuss such a result in forthcoming work with Siddharth. 
    
    We note that the formula for the $\THR$ of Thom spectra \cite[Corollary 7.1.3]{hahn2024equivariantnonabelianpoincareduality} requires that the base space of the Thom spectrum is connected, so it cannot be used to compute $\THR(\MUP_{\R})$. Even if one can get rid of the connectedness assumption, it would yield $\THR(\MUP_{\R}) \simeq \MUP_{\R} \otimes \Sigma_+^{\infty} \B^{\sigma}\BUP_{\R}$, which is different from our computation in \cref{theorem: THR KUR}. This is in parallel to the non-equivariant phenomenon 
    \[ \THH(\MUP) \simeq \MUP \otimes \Sigma_+^{\infty} \SU \simeq \MUP \otimes \Sigma_+^{\infty} \mathrm{U} \] 
    using Snaith resp.~the $\THH$ formula for Thom spectra. See \cite[Example 7.13]{rasekhStonekValenzuela2022thom}.
\end{remark}

A similar observation can be made for $\MOP$. The formula for the $\THH$ of Thom spectra yields
\[ \THH(\MOP) \simeq \MOP \otimes \Sigma_+^{\infty} \mathrm{BBOP} \simeq \MOP \otimes \Sigma_+^{\infty} \mathrm{O}, \]
while we can also obtain:

\begin{proposition} \label{prop: THH MOP}
    There is an equivalence $\THH(\MOP) \simeq \MOP \otimes \Sigma_+^{\infty} \B^2 \mathrm{O}$ of spectra.
\end{proposition}

\begin{proof}
    This follows the same strategy as in \cref{theorem: THR KUR} using \cref{corollary: THH commutes with inverting}. The main ingredient is the real Snaith theorem, i.e.~an equivalence $\MOP \simeq \Sigma_+^{\infty}\BO[\eta^{-1}]$ of $\E_1$-ring spectra (\cref{corollary: real Snaith}).
\end{proof}

\bibliographystyle{alpha}
\bibliography{main}

\Addresses

\end{document}